\documentclass[11pt, article]{amsart}

\usepackage{amsmath,amssymb,amscd,amsfonts,comment}
\usepackage[textwidth=360pt,textheight=615pt]{geometry}
\usepackage{xcolor}

\newtheorem{thm}{Theorem}[section]
\newtheorem{lem}[thm]{Lemma}
\newtheorem{rem}[thm]{Remark}
\newtheorem{prop}[thm]{Proposition}

\newtheorem{claim}[thm]{Claim}

\newcommand\ct{{\text{\rm ct}}}
\newcommand\lct{{\text{\rm lct}}}

\newcommand\CC{{\mathbb{C}}}

\newcommand\RR{{\mathbb{R}}}

\newcommand\cX{{\mathcal{X}}}
\newcommand\cY{{\mathcal{Y}}}
\newcommand\cZ{{\mathcal{Z}}}

\newcommand\bR{{\mathbb R}}
\newcommand\bZ{{\mathbb Z}}
\newcommand\bC{{\mathbb C}}
\newcommand\bQ{{\mathbb Q}}
\newcommand\bN{{\mathbb N}}

\begin{document}
\title{Classification of threefold canonical thresholds}
\author{Jheng-Jie Chen}
\address{\rm Department of Mathematics, National Central University, Taoyuan City, 320, Taiwan}
\email{jhengjie@math.ncu.edu.tw}
\author{Jiun-Cheng Chen}
\address{\rm Department of Mathematics\\ Third General Building\\ National Tsing Hua University\\ No. 101 Sec 2 Kuang Fu Road\\ Hsinchu, 30043\\ Taiwan}
\email{jcchen@math.nthu.edu.tw}
\author{Hung-Yi Wu}
\address{\rm Department of Mathematics, National Central University, Taoyuan City, 320, Taiwan}
\email{112281002@cc.ncu.edu.tw}

\maketitle

\begin{abstract}  

 We show that the set $\mathcal{T}_{3, \textup{sm}}^{\textup{can}}$ of smooth threefold canonical thresholds coincides with $\mathcal{T}_{2, \textup{sm}}^{\textup{lc}}=\mathcal{HT}_{2}$, where $\mathcal{HT}_{2}$ is the $2$-dimensional hypersurface log canonical thresholds characterized by Kuwata \cite{K99a, K99b}. We classify the set $\mathcal{T}_{3}^{\textup{can}}$ of threefold canonical thresholds. 
More precisely, we prove   $\mathcal{T}_{3}^{\textup{can}}= \{0\} \cup \{\frac{4}{5}\} \cup \mathcal{T}_{3, \textup{sm}}^{\textup{can}}$. 
\end{abstract}

\section{introduction}

We work over the complex number field $\mathbb{C}$.

For a log canonical pair $(X,S)$, the log canonical threshold $$\lct(X,S):=\sup\{t\in \bR\ |\ (X,tS) \textup{ is log canonical}\}$$ is a natural and fundamental invariant which measures the complexity of singularities in algebraic geometry (see \cite{Kol97, Kol08}). 
The set of $n$-dimensional log canonical thresholds  $$\mathcal{T}_{n}^{\textup{lc}}:=\{\lct(X,S) | \dim X = n \}$$
has very interesting properties related to the minimal model program (cf. \cite{Sho93, MP04, Bir07,HMX, LMX24}). Note that it is an interesting and very hard question to describe the set $\mathcal{T}_{n}^{\textup{lc}}$ explicitly. Certain classifications in $n\leq 3$ were investigated (cf. \cite{Ale93, Sho93, K99a, K99b, Kol08}).  

We consider an analogous notion. Let $X \ni P$ be a germ of complex algebraic variety with at worst canonical singularities and $S \ni P$ be a prime $\mathbb{Q}$-Cartier divisor. The canonical threshold of the pair $(X,S)$ is defined as 
$$\ct(X,S):=\sup\{t\in \bR\ |\ \textup{the pair }(X,tS) \textup{ is canonical}\}.$$ For every natural  number $n$, the set of canonical thresholds is defined by
$$\mathcal{T}_{n}^{\textup{can}} := \{\ct(X, S) | \dim X = n \}.$$ 

It is known that  $\mathcal{T}_{2}^{\textup{can}}=\{0\} \cup \{\frac{1}{k}\}_{k\in \mathbb{N}}$ where $\mathbb{N}$ denotes the set of positive integers.
In this paper, we focus on the classifications in the case $n=3$. 

Canonical thresholds appear naturally and play crucial roles in the Sarkisov program. 
Recall that every canonical threshold  $\ct(X,S)\in \mathcal{T}^{\textup{can}}_{3}$ is computed by some divisorial contraction $\sigma\colon Y\to X$ (cf. \cite{Corti} or \cite{Matsuki}).
As the remarkable works of classifications of threefold divisorial contractions to points are completely investigated by Hayakawa, Kawakita, Kawamata, Mori,  Yamamoto, and many others (cf. \cite{Mori82, Ka,Haya99,Haya00,Kawakita01,Kawakita02,Kawakita05,Yama18}), it is then a natural question to describe the set $\mathcal{T}^{\textup{can}}_3$ explicitly.

We collect some known results toward the description of the set $\mathcal{T}^{\textup{can}}_3$ as follows.
 Note that certain numbers are in this set, e.g. $\min \{\frac{1}{\alpha}+\frac{1}{\beta}, 1\}$ (by a result of Stepanov \cite{Stepanov}) for all positive integers $\alpha$ and $\beta$. It is also known that $\frac{4}{5} \in \mathcal{T}_{3}^{\textup{can}}$ \cite{Prok08}. Consider the subset $$\mathcal{T}_{3, \textup{sm}}^{\textup{can}}:=\{\ct(X,S)\ |\ \dim X=3, X\textup{ is smooth}\} \subseteq \mathcal{T}_{3}^{\textup{can}}.$$ 
Prokhorov proves that $\mathcal{T}_{3}^{\textup{can}} \cap (\frac{5}{6},1)= \emptyset$
as well as $\ct(X,S) \leq \frac{4}{5}$ when $X$ is singular \cite{Prok08}.
Stepanov proves that $\mathcal{T}_{3}^{\textup{can}} \cap (\frac{4}{5},\frac{5}{6}) = \emptyset$ \cite{Stepanov}.  
Another interesting question concerning the set $\mathcal{T}_{3}^{\textup{can}}$  is to determine if it  satisfies the ACC. 
Stepanov obtains that $\mathcal{T}_{3, \textup{sm}}^{\textup{can}}$ satisfies the ACC and establishes the explicit formula for 
$\ct_P(X,S)$ when $P\in S$ is a Brieskorn singularity in \cite{Stepanov}. 
Applying Stepanov's argument, 
the first named author proves the ACC for $\mathcal{T}_{3}^{\textup{can}}$ \cite{3ct}.
He also shows that $\mathcal{T}_{3}^{\textup{can}}\cap (\frac{1}{2},1)$ coincides with $\{ \frac{1}{2}+\frac{1}{p}\}_{p \in \mathbb{Z}_{\ge 3}} \cup \{ \frac{4}{5}\}$. Moreover, Han, Liu and Luo and the first named author independently prove that the accumulation points of $\mathcal{T}_{3}^{\textup{can}}$ coincides with $\mathcal{T}_{2}^{\textup{can}}\setminus\{1\}$ and generalized the ACC to pairs in \cite{HLL, 3ct}.

We now state the main results of this paper. The first result is an explicit description of the set $\mathcal{T}_{3, \textup{sm}}^{\textup{can}}$ of threefold canonical thresholds. We also discover $\mathcal{T}^{\textup{lc}}_{2,\textup{sm}}= \mathcal{T}^{\textup{can}}_{3,\textup{sm}}$. It is interesting to compare this with the well-known theorem  that the set of limit points of $\mathcal{T}^{\textup{lc}}_{n+1, \textup{sm}}$ is $\mathcal{T}^{\textup{lc}}_{n, \textup{sm}}$ \cite{dFM} \cite{Kol08}.
\begin{thm}\label{classfysmcintroduction}(= Theorem \ref{classfysmct})
The set $\mathcal{T}^{\textup{can}}_{3,\textup{sm}}$ consists of $C\cap [0,1]$ where $C$ is the following set 
$$\left\{\frac{\alpha+\beta}{p_1\alpha+p_2\beta} \Bigg| \begin{array}{l} \alpha, \beta, p_2\in \mathbb{N} \textup{ and } p_1 \in \mathbb{Z}_{\ge 0 } \textup{ such that }\alpha\leq \beta,\ \\ \gcd(\alpha,\beta)=1 \textup{ and either } p_2\geq \max\{\alpha, p_1\} \textup{ or }p_2=p_1  \end{array} \right\}.$$ 
In particular, we have $$\mathcal{T}^{\textup{can}}_{3,\textup{sm}}=\mathcal{T}_{2,\textup{sm}}^{\textup{lc}} (= \mathcal{HT}_2 ).$$
\end{thm}

The next result characterizes the set of $3$-dimensional canonical thresholds $\mathcal{T}^{\textup{can}}_{3}$.
It is plausible that most of the canonical thresholds should come form the smooth case already. 
We show that $\frac{4}{5}$ is the only non-trivial exception.
\begin{thm} \label{classfy3ctintroduction} (=Theorem \ref{classfy3ct})
We have   $\mathcal{T}^{\textup{can}}_{3} = \{0\} \cup \{\frac{4}{5}\}\cup \mathcal{T}_{3, \textup{sm}}^{\textup{can}}.$   
\end{thm}

In what follows, we explain the proof of Theorems \ref{classfysmcintroduction}. 
Suppose that $\ct(X,S)\in \mathcal{T}^{\textup{can}}_{3,\textup{sm}}$ is a canonical threshold. It is known that $\ct(X,S)\leq 1$ and $\ct(X,S)$ can be computed as a weighted blow up $Y\to X=\hat{\mathbb{C}}^3$ at the origin of $\hat{\mathbb{C}}^3$ with weights $w:=(1,\alpha,\beta)$ for some positive relative prime integers $\alpha$ and $\beta$ by Kawakita \cite{Kawakita01} such that $\ct(X,S)=\frac{\alpha+\beta}{m}$ where $m=w(f)$ is the weighted multiplicity of the defining convergent power series $f$ of the prime divisor $S$. For simplicity, we assume that $1<\alpha<\beta$. 
$Y$ is then a union of three affine open subsets $U_1, U_2$ and $U_3$ such that $U_1$ is smooth, 
\[U_2\simeq \hat{\mathbb{C}}^3/\frac{1}{\alpha}(-s,s, 1) \textup{ and }U_3\simeq \hat{\mathbb{C}}^3/\frac{1}{\beta}(t,1, -t),
\] where $s$ and $t$ are the positive integers with $\alpha t=\beta s+1$. Denote by $\bar{s}:=\alpha-s$ and $\bar{t}:=\beta-t$ so that $\alpha\bar{t}=\beta\bar{s}-1$. 
Note that there are two weights $w_2=(1,s,t)$ and $w_3=(1,\bar{s},\bar{t})$ over $X$ corresponding to Kawamata blow ups at the origins of $U_2$ and $U_3$ respectively.
Comparing $w$ with two auxiliary weights $w_2$ and $w_3$, \cite[Lemma 2.1]{3ct} yields that 
\begin{align}
    \lfloor  \frac{s+t}{\alpha+\beta} m \rfloor  \ge   \lceil \frac{s}{\alpha} m \rceil\textup{ and }\lfloor  \frac{\bar{s}+\bar{t}}{\alpha+\beta} m \rfloor\geq \lceil \frac{\bar{t}}{\beta} m \rceil. \label{SNI1}
\end{align}

By \cite[Proposition 3.3]{3ct}, we have  $m\geq \alpha\beta$. In particular, there exist non-negative integers $p_1$ and $p_2$ with $m=p_1\alpha+p_2\beta$ and $p_1<\beta$.
From the inequalities in (\ref{SNI1}) and the above identities $\alpha t=\beta s+1$ and $ \alpha \bar{t}=\beta \bar{s}-1$, we obtain $p_2\geq p_1$ (resp. $p_2\geq \alpha$ if $p_1\neq p_2$). Thus, $\ct(X,S)\in C\cap [0,1]$. Conversely, given a number $\frac{\alpha+\beta}{p_1\alpha+p_2\beta}\in C\cap [0,1]$ where $\alpha,\beta,p_2$ are positive integers and $p_1$ is a non-negative integer with $\alpha\leq \beta$ and $\gcd(\alpha,\beta)=1$ such that either $p_2 \geq \max \{\alpha, p_1 \}$ or $p_1=p_2$. We assume $\alpha>1$ for simplicity. Denote by $m=p_1\alpha+p_2\beta$ and weights $w=(1,\alpha,\beta), w_2=(1,s,t)$ and $ w_3=(1,\bar{s},\bar{t})$. We are able to construct a prime divisor  $S:=\{f=0\}$ (near the origin of $\hat{\mathbb{C}}^3$) satisfying the following two conditions: \begin{itemize}
    \item the proper transform $S_Y$ of $S$ in $Y$ is smooth except probably the origins of $U_2$ and $U_3$ (near the exceptional divisor of weighted blow up $\sigma\colon Y\to X=\hat{\mathbb{C}}^3$ with weights $w$);
    \item we have the inequalities 
\begin{align*} \frac{s+t}{\alpha+\beta}m\geq w_2(f) \textup{ and }\frac{\bar{s}+\bar{t}}{\alpha+\beta}m\geq w_3(f)  \label{SNI2}\end{align*} where $m=w(f), w_2(f)$ and $w_3(f)$ are the weighted multiplicities.
\end{itemize}
Thanks to computations of canonical thresholds for terminal cyclic quotient singularities studied by Kawamata \cite[Lemma 7]{Ka}, it follows that $\frac{\alpha+\beta}{p_1\alpha+p_2\beta}\in \mathcal{T}^{\textup{can}}_{3,\textup{sm}}$ (see Lemma \ref{compute smct}). Therefore, we have established $\mathcal{T}^{\textup{can}}_{3,\textup{sm}}=C\cap [0,1]$. It is then straightforward to check that $C\cap [0,1]$ coincides with $\mathcal{HT}_2$ where $\mathcal{HT}_2$ is the set of $2$-dimensional hypersurface log canonical thresholds explicitly classified by Kuwata \cite{K99a, K99b}. Thus, Theorem \ref{classfysmcintroduction} follows.

Theorem \ref{classfy3ctintroduction} is basically derived from the argument of Theorem \ref{classfysmcintroduction} and the classifications of divisorial contractions by Kawakita \cite{Kawakita05} which we explain as follows. Recall that for every projective threefold $X$ with at worst $\bQ$-factorial terminal singularities,  $\ct(X,S)\in \mathcal{T}^{\textup{can}}_{3}$ is obtained as $\ct(X,S)=\frac{a}{m}$ for some divisorial extraction $\sigma: Y \to X$ extracting only one irreducible divisor $E$ and 
\[K_Y=\sigma^*K_X+\frac{a}{n} E \textup{ and } S_Y=\sigma^*S-\frac{m}{n}E\] where $\frac{a}{n}$ (resp. $\frac{m}{n}$) is the discrepancy (resp. multiplicity) and $n$ is the index of the center $\sigma(E)$ (cf \cite{Corti} or \cite{Matsuki}). We say that $\ct(X,S)$ is computed by $\sigma$ and denote the weighted discrepancy and weighted discrepancy by $a$ and $m$, respectively. 
As in \cite{3ct}, consider \begin{equation*}
 \mathcal{T}_{3, *,\geq 5}^{\textup{can}}:=\left\{\ct(X,S) \Bigg| \begin{array}{l}  \ct(X,S)  \text{ is computed by } \sigma: Y \to X \textup{ with weighted }\\
                                      \text{ discrepancy }a  \ge 5 \text{ contracting a divisor } E  \\
                                         \textup{to a closed point } \sigma(E)=P\in X \text{ of type } *  \end{array} \right\}
 \end{equation*}
where the type $*$ can be $cA$ (resp. $cA/m$, $cD$ or $cD/2$) if $P \in X$ is a singular point of type $cA$ (resp. $cA/m$, $cD$ or $cD/2$).
According to classifications of threefold divisorial contractions by Kawakita in \cite{Kawakita05}, 
we have the following decomposition :
 $$\label{decomposition}
 \mathcal{T}_{3}^{\textup{can}} = \{0\} \cup \aleph_4 \cup \mathcal{T}_{3,\textup{sm}}^{\textup{can}} \cup \mathcal{T}_{3, cA,\geq 5}^{\textup{can}} \cup \mathcal{T}_{3, cA/m,\geq 5}^{\textup{can}} \cup \mathcal{T}_{3,cD,\geq 5}^{\textup{can}} \cup \mathcal{T}_{3, cD/2,\geq 5}^{\textup{can}} \eqno{(2)}$$
where $\aleph_4:=\mathcal{T}_{3}^{\textup{can}}\cap \{\frac{a}{m}\}_{a,m\in \bN,\ a\leq \max\{4,m\}}$.  
From Theorem \ref{classfysmcintroduction} and a computation of Prokhorov \cite{Prok08}, one has $\aleph_4 \cup \mathcal{T}_{3,\textup{sm}}^{\textup{can}}=\{\frac{4}{5}\}\cup \mathcal{T}_{3,\textup{sm}}^{\textup{can}}$ (see Remark \ref{quickremark1}). Theorem \ref{classfy3ctintroduction} then follows from the inclusions $\mathcal{T}_{3,*,\ge 5}^{\textup{can}}\subseteq C\cap [0,\frac{4}{5}]$ for $*=cA, cA/n, cD$ and $cD/2$ in Proposition \ref{possible 3ctcA}, \ref{possible 3ctcA/n new}, \ref{classfycDct} and \ref{classfycD2ct}. Those propositions utilize the argument of the inclusion $\mathcal{T}_{3,\textup{sm}}^{\textup{can}}\subseteq C\cap [0,1]$ in Theorem \ref{classfysmcintroduction} with more careful considerations.


\subsection{Acknowledgement} The first named author was partially supported by the National Science and Technology Council of Taiwan (Grant Numbers: 112-2115-M-008 -006 -MY2 and 113-2123-M-002-019-). We would like to thank Professors Jungkai Chen, Hsueh-Yung Lin, Jihao Liu, and Dr. Iacopo Brivio for some helpful conversations and encouragement.

\section{Classification of  $\mathcal{T}^{\textup{can}}_{3,\textup{sm}}$}
The aim of this section is to classify the set $\mathcal{T}^{\textup{can}}_{3,\textup{sm}}$ of smooth threefold canonical thresholds. We begin with the following lemma using the computation of canonical threshold for terminal cyclic quotient singularity by Kawamata in \cite{Ka}.

\begin{lem}\label{compute smct}
    Let $\mu_1\colon Y\to X=\hat{\mathbb{C}}^3\ni P$ be a weighted blow up with weights $w=(1,\alpha,\beta)$ and exceptional divisor $E_1$ where $\beta>1$. Let $S$ be a prime divisor of $X$ defined by a power series $f$. 
    Let $\bar{s}$ and $\bar{t}$ be positive integers with $\alpha \bar{t
}=\beta \bar{s}-1$ (resp. $s$ and $t$ positive integers with $\alpha t
=\beta s+1$ if $\alpha>1$). Define $m=w(f)$ and $m_3:=w_3(f)$ where $w_3=(1,\bar{s},\bar{t})$ (resp. $m_2:=w_2(f)$ where $w_2=(1,s,t)$). 
Suppose that $\frac{\bar{s}+\bar{t}}{\alpha+\beta}m\geq m_3$ (resp. $\frac{s+t}{\alpha+\beta}m\geq m_2$ if $\alpha>1$) and the proper transform $S_Y$ is non-singular on $Y$ near $E_1$ except probably the origin of $U_3$ (resp. $U_2$ if $\alpha>1$).
    Then $\ct_P(X,S)=\min\{\frac{\alpha+\beta}{m},1\}$.
\end{lem}
\begin{proof}The proof actually follows from \cite[Lemma 6 and argument of Lemma 7]{Ka}.

Since $\ct(X,S)\in [0,1]$, we may assume $\frac{\alpha+\beta}{m}<1$.  
Suppose that $\alpha>1$.
Recall that $Y$ is a union of three open charts $U_1\simeq \hat{\mathbb{C}}^3$, $U_2\simeq \mathbb{C}^3/\frac{1}{\alpha}(-1,1,-\beta)$ and $U_3\simeq \mathbb{C}^3/\frac{1}{\beta}(-1,-\alpha,1)$.  Let $\mu_2:Y_2\to Y$ be the Kawamata blow up (with weights $v_2=\frac{1}{\alpha}(\bar{s}, s, 1)$) at the origin of $U_2$ with exceptional divisor $E_2$. 
Let $\mu_3:Y_3\to Y_2$ be the Kawamata blow up at the origin of $U_3$ with weights $v_3=\frac{1}{\beta}(t,1,\bar{t})$ and exceptional divisor $E_3$.  Let $n\geq n_0\geq 1$ be integers such that $\mu_i:Y_i\to Y_{i-1}$ are Kawamata blow ups for $i=1,2,...,n_0$ (resp. usual blow up at smooth point or along smooth curve for $i=n_0+1,...,n$), $Y_{n_0}$ is non-singular and the composition $\sigma:=\mu_n\circ\cdots\circ\mu_2\circ\mu_1 \colon Y_n\to X$ is a log resolution of $(X,S)$ near $P\in X$. Note that we set $Y_1:=Y$ and $Y_0:=X$.
Write $K_{Y_n}=\sigma^*K_X+\sum_{j=1}^n a_j E_j$ and $\sigma^*S=S_{Y_n}+\sum_{j=1}^n m_j E_j$ with exceptional divisor $E_j$ of $\mu_j:Y_j\to Y_{j-1}$ where $m_1:=m=w(f)$. For every $j=1,...,n$, denote by $\sigma_j:=\mu_n\circ\cdots\circ\mu_{j+1}:Y_n\to Y_j$ the induced morphism and write $$K_{Y_j}=\mu_j^*K_{Y_{j-1}}+\bar{a}_jE_j,\ \  \mu_j^*S_{Y_{j-1}}=S_{Y_j}+\bar{m}_jE_j \textup{ and }\sigma_j^* E_j=E_j+\sum_{j<k} \alpha_{jk}E_k$$ for non-negative rational numbers $\alpha_{jk}$. Then we have $a_j=\bar{a}_j+\sum_{i<j}a_i \alpha_{ij}$ and $m_j=\bar{m}_j+\sum_{i<j} m_i \alpha_{ij}$. 

As $\mu_3:Y_3\to Y_2$ is a weighted blow up with weights $v_3=\frac{1}{\beta}(t,1,\bar{t})$ and $w_3=\frac{\bar{t}}{\beta}w+\frac{1}{\beta}(t,1,0)$, we see $\alpha_{13}=\frac{\bar{t}}{\beta}$ and $\bar{a}_3=\frac{1}{\beta}$. 
From  
$$\frac{\bar{s}+\bar{t}}{m_3}=\frac{a_3}{m_3}=\frac{\bar{a}_3+\alpha_{13}a_1}{\bar{m}_3+\alpha_{13}m}=\frac{\bar{a}_3+\alpha_{13}(\alpha+\beta)}{\bar{m}_3+\alpha_{13}m}$$ and $\alpha_{13}\geq 0$,
it is easy to see that the assumption $\frac{\bar{s}+\bar{t}}{\alpha+\beta}m\geq m_3$ is equivalent to $\frac{\bar{a}_3}{\bar{m}_3}\geq \frac{\alpha+\beta}{m}=\frac{a_1}{m_1}$.
Similarly, the assumption $\frac{s+t}{\alpha+\beta}m\geq m_2$ is equivalent to $\frac{\bar{a}_2}{\bar{m}_2}\geq \frac{\alpha+\beta}{m}=\frac{a_1}{m_1}$.

It remains to show $\frac{a_j}{m_j}\geq \frac{a_1}{m_1}$ for every $j=1,...,n$. It is obvious when $j=1$. 
From the assumptions  $\frac{\bar{s}+\bar{t}}{\alpha+\beta}m\geq m_3$ and $\frac{s+t}{\alpha+\beta}m\geq m_2$, we may assume that $j\geq 4$. By induction hypothesis, we assume that $\frac{a_k}{m_k}\geq \frac{a_1}{m_1}$ for $k=1,...,j-1$.
Denote by $\sigma_3(E_j)$ the center of $E_j$ on $Y_3$. 

\noindent 
{\bf Case 1:} $\sigma_3(E_j)$ is contained in $E_3$. 
By \cite[Lemma 6]{Ka}, one has $\frac{\bar{a}_j}{\bar{m}_j}\geq \frac{\bar{a}_3}{\bar{m}_3}$. Thus, we see $\frac{\bar{a}_j}{\bar{m}_j}\geq \frac{\bar{a}_3}{\bar{m}_3}\geq \frac{a_1}{m_1}$. In particular, we observe 
$$\frac{a_j}{m_j}=\frac{\bar{a}_j+\sum_{i<j}a_i \alpha_{ij}}{\bar{m}_j+\sum_{i<j} m_i \alpha_{ij}}\geq \frac{a_1+\sum_{i<j}a_i \alpha_{ij}}{m_1+\sum_{i<j}m_i \alpha_{ij}}\geq \frac{a_1+\sum_{i<j}a_1 \alpha_{ij}}{m_1+\sum_{i<j}m_1 \alpha_{ij}}=\frac{a_1}{m_1},$$ where the last inequality follows from induction hypothesis.

\noindent 
{\bf Case 2:} $\sigma_3(E_j)$ is contained in $E_2$. We may apply the same argument in Case 1 to obtain $\frac{a_j}{m_j}\geq \frac{a_1}{m_1}$.

\noindent 
{\bf Case 3:} $\sigma_3(E_j)$ is contained in $E_1$ and not contained in $E_2\cup E_3$. From the assumption that $S_Y$ is non-singular on $Y$ (near $E_1$) except the origins of $U_2$ and $U_3$, we see that $\bar{m}_j\in \{1,0\}$ and $\bar{a}_j=1$ or $2$. Thus $\frac{\bar{a}_j}{\bar{m}_j}=\bar{a}_j\geq 1 > \frac{\alpha+\beta}{m}=\frac{a_1}{m_1}$ or $\bar{m}_j=0$. In particular, $\frac{a_j}{m_j}\geq \frac{a_1}{m_1}$ by induction hypothesis.

Suppose that $\alpha=1$. Then $U_2\simeq \mathbb{C}^3$ is non-singular. Replacing $j$ by $j-1$ for $j\geq 3$ above. We divide it into two cases that $\sigma_2(E_j)$ is contained in $E_2$ or $E_1\setminus E_2$. The above arguments yield the same inequality $\frac{a_j}{m_j}\geq \frac{a_1}{m_1}$. We finish the proof of Lemma \ref{compute smct}.
\end{proof}

Let $f$ be a non-zero holomorphic function near $0 \in \CC ^2$. Recall that $c_{0}(\CC ^2,f):= \text{sup}\{c | \;|f|^{-c} \; \text{is locally $L ^2$ near $0$}\}$. Following Kuwata \cite{K99a} and \cite{K99b}, we define  $$\mathcal{HT}_2:=
\{ c_{0}(\CC^{2},f): f \neq 0   \text{ is holomorphic near $0\in \CC^{2}$} \}.$$

\begin{thm}\label{classfysmct}
The set $\mathcal{T}^{\textup{can}}_{3,\textup{sm}}$ consists of $C\cap [0,1]$ where $C$ is the following set
$$\left\{\frac{\alpha+\beta}{p_1\alpha+p_2\beta} \Bigg| \begin{array}{l} \alpha, \beta, p_2\in \mathbb{N} \textup{ and } p_1 \in \mathbb{Z}_{\ge 0 } \textup{ such that }\alpha\leq \beta,\ \\ \gcd(\alpha,\beta)=1 \textup{ and either } p_2\geq \max\{\alpha, p_1\} \textup{ or }p_2=p_1  \end{array} \right\}.$$ In particular, $$\mathcal{T}^{\textup{can}}_{3,\textup{sm}}=\mathcal{T}_{2,\textup{sm}}^{\textup{lc}} (= \mathcal{HT}_2 ).$$
\end{thm}
\begin{proof}

We first show  $\mathcal{T}_{3,\textup{sm}}^{\textup{can}} \subseteq C\cap [0,1]$.
Suppose that $\ct(X,S)\in \mathcal{T}_{3,\textup{sm}}^{\textup{can}}$. By \cite[(2.10) Proposition-definition]{Corti} and \cite{Kawakita01}, the canonical threshold  $\ct(X,S)=\frac{\alpha+\beta}{m}$ is realized by a weighted blow up $\sigma: Y \to X=\hat{\bC}^3$ with weights $w=(1,\alpha,\beta)$ where $\alpha$ and $\beta$ are relative prime integers with $1\leq \alpha\leq \beta$ and $m$ is the weighted multiplicity of $S$ with respect to $\sigma$.

In the case $\alpha=\beta=1$, every positive integer is of the form $p_1\alpha+p_2\beta$ for some non-negative integers $p_1$ and $p_2$. 
Thus, we may assume that $\alpha<\beta$. 
As $\gcd(\alpha,\beta)=1$, there exists positive integers $\bar{s}$ and $\bar{t}$ with $\alpha \bar{t
}=\beta \bar{s}-1$ (resp, positive integers $s$ and $t$ with $\alpha t
=\beta s+1$ if $\alpha>1$). Define $w_3=(1,\bar{s},\bar{t})$ and $\sigma_3$ to be the weighted blow up at the origin of $X=\hat{\mathbb{C}}^3$ with the weights $w_3$. Similarly, define $w_2=(1,s,t)$ and $\sigma_2$ to be the weighted blow up at the origin of $X=\hat{\mathbb{C}}^3$ with the weights $w_2$ if $\alpha>1$. 
Let $m_3$ (resp. $m_2$ if $\alpha>1$) denote the weighted multiplicity of $S$ with respect to $\sigma_3$ (resp. $\sigma_2$ if $\alpha>1$). From \cite[Lemma 2.1]{3ct}, one has 
$$ \lfloor  \frac{\bar{s}+\bar{t}}{\alpha+\beta} m \rfloor \ge m_3 \ge    \lceil \frac{\bar{t}}{\beta} m \rceil.$$
Similarly, in the case $\alpha>1$, we have 
$$\lfloor  \frac{s+t}{\alpha+\beta} m \rfloor   \ge m_2 \ge   \lceil \frac{s}{\alpha} m \rceil.$$

By \cite[Proposition 3.3]{3ct} and $\ct(X,S)\leq 1$, we see $m\geq \alpha\beta$. As $m$ is an integer greater than $\alpha\beta-\alpha-\beta$, one may write $m=p_1\alpha+p_2\beta$ for some non-negative integers $p_1$ and $p_2$. Rewriting  $$m=p_1\alpha+p_2\beta=(p_1-\beta\lfloor \frac{p_1}{\beta}\rfloor)\alpha+(p_2+\lfloor \frac{p_1}{\beta}\rfloor \alpha)\beta,$$ one may assume that $p_1<\beta$. It remains to show the following two claims.
\begin{claim}\label{p2gp1}
 We have  $p_2\ge p_1$.
\end{claim}
\begin{proof}[Proof of the Claim]
Since $\alpha \bar{t}=\beta\bar{s}-1$ and we assume that $p_1<\beta$, one has 
$$\lceil \frac{\bar{t}}{\beta} m \rceil=\lceil \frac{\bar{t}}{\beta} (p_1\alpha+p_2\beta) \rceil= \lceil p_1\bar{s}+p_2\bar{t}- \frac{p_1}{\beta}\rceil=p_1\bar{s}+p_2\bar{t}.$$
As $\lfloor\frac{\bar{s}+\bar{t}}{\alpha+\beta}m\rfloor\ge m_3\ge  \lceil \frac{\bar{t}}{\beta} m \rceil$ where 
$$\lfloor \frac{\bar{s}+\bar{t}}{\alpha+\beta}m\rfloor=\lfloor \frac{\bar{s}+\bar{t}}{\alpha+\beta}(p_1\alpha+p_2\beta)\rfloor=p_1\bar{s}+p_2\bar{t}+\lfloor \frac{p_2-p_1}{\alpha+\beta}\rfloor, $$  
we conclude $p_2\geq p_1$.    
\end{proof}

\begin{claim}
 Either  $p_2\ge \alpha$ or $p_1=p_2$.
\end{claim}
\begin{proof}[Proof of the Claim]
The argument is similar to the claim above. We may assume that $\alpha>1$. 
Since $\alpha t=\beta s+1$, we have 
\begin{align*}
\lfloor\frac{s+t}{\alpha+\beta}m\rfloor&=\lfloor\frac{s+t}{\alpha+\beta}(p_1\alpha+p_2\beta)\rfloor= p_1s+p_2t-\lceil \frac{p_2-p_1}{\alpha+\beta} \rceil \\ 
\textup{ and \ \ } \lceil \frac{s}{\alpha} m \rceil&=\lceil \frac{s}{\alpha} (p_1\alpha+p_2\beta) \rceil=p_1s+p_2t-\lfloor \frac{p_2}{\alpha} \rfloor.
\end{align*}
As we have $\lfloor  \frac{s+t}{\alpha+\beta} m \rfloor  \ge m_2 \ge   \lceil \frac{s}{\alpha} m \rceil,$
one sees $$ \frac{p_2}{\alpha}\geq\lfloor\frac{p_2}{\alpha} \rfloor\ge \lceil \frac{p_2-p_1}{\alpha+\beta} \rceil\ge 0,$$ where last inequality follows from Claim \ref{p2gp1} above.
Thus, we conclude $p_2\geq \alpha$ or $p_1=p_2$. 
\end{proof}

Conversely, suppose that we are given $\alpha, \beta, p_2\in \mathbb{N}$ and $p_1 \in \mathbb{Z}_{\ge 0 }$ with  $\alpha\leq \beta,\ \gcd(\alpha,\beta)=1$ satisfying either $p_2\geq \max\{\alpha,p_1\}$ or $p_1=p_2$. We shall show that $\frac{\alpha+\beta}{p_1\alpha+p_2\beta}\in \mathcal{T}^{\textup{can}}_{3,\textup{sm}}$ if $\frac{\alpha+\beta}{p_1\alpha+p_2\beta}\leq 1$. 


 If $\alpha=\beta=1$, we take $S=(x^m+y^m+z^m=0)$ and $\sigma:Y\to X=\hat{\mathbb{C}}^3$ to be the smooth blow up at the origin where $m$ is a positive integer. As $(Y, S_Y)$ is log smooth for $m\geq 2$ (resp. $(\hat{\mathbb{C}}^3,S)$ is log smooth when $m=1$), we see $\ct(X,S)=\min\{\frac{2}{m},1\}$ by \cite[Theorem 3.6]{Stepanov}. 

From now on, we may assume $\beta>1$.
Let $\sigma\colon Y\to X=\hat{\mathbb{C}}^3$ be the weighted blow up with weights $w=(1,\alpha,\beta)$. Define $\bar{s}$ and $ \bar{t}$ to be positive integers with $\alpha \bar{t}=\beta\bar{s}-1$. Similarly, in the case $\alpha>1$, we define positive integers $s=\alpha-\bar{s}$ and $t=\beta-\bar{t}$ so we have $\alpha t=\beta s+1$. 
Note that $Y$ is covered by three affine open subsets $U_1, U_2, U_3$ where $U_1\simeq \hat{\mathbb{C}}^3$ and
 \begin{align*}
U_3&\simeq \hat{\mathbb{C}}^3/\frac{1}{\beta}(-1,-\alpha, 1)\simeq \hat{\mathbb{C}}^3/\frac{1}{\beta}(t,1, -t) \textup{ and } \\ U_2&\simeq \hat{\mathbb{C}}^3/\frac{1}{\alpha}(-1,1, -\beta)\simeq \hat{\mathbb{C}}^3/\frac{1}{\alpha}(-s,s, 1) \textup{ if }\alpha>1.
\end{align*}
Define $$v_2=\frac{1}{\alpha}(\bar{s}, s, 1),\ v_3=\frac{1}{\beta}(t,1, \bar{t}),\ w_2=\frac{s}{\alpha}w+\frac{1}{\alpha}(\bar{s},0,1), w_3= \frac{\bar{t}}{\beta}w+\frac{1}{\beta}(t,1,0).$$
Then $w_2=(1,s,t)$, $w_3=(1, \bar{s}, \bar{t})$ 
and the weighted blow up $\sigma_3':Z_3\to U_3$ (resp. $\sigma_2':Z_2\to U_2$) at the origin of $U_3$ (resp. $U_2$) with weights $v_3$ (resp. $v_2$ if $\alpha>1$) is Kawamata blow up. 

Define $m:=p_1\alpha+p_2\beta$ and express $$p_2=\alpha\lfloor \frac{p_2}{\alpha}\rfloor +q$$ where $q\in [0,\alpha-1]$ is an integer.

\noindent 
\textbf{Case 1:  $\alpha| p_2$}. 
Define $S$ to be the Cartier divisor $f=0$ in $\hat{\mathbb{C}}^3$ where $$f=x^{m}+y^{p_1}z^{p_2}+y^{\frac{p_2}{\alpha}\beta+p_1}+z^{m}.$$ 
Note that $m=w(f)$, and $$w_3(f)=\min\{w_3(x^m),w_3(y^{p_1}z^{p_2}),w_3(y^{\frac{p_2}{\alpha}\beta+p_1}), w_3(z^{m})\}=w_3(y^{p_1}z^{p_2})=p_1\bar{s}+p_2\bar{t}$$ and 
$S$ is only singular at the origin if $p_1+p_2\geq 2$. In particular, the divisor $S$ is prime.  
In the open chart $U_1\simeq \hat{\mathbb{C}}^3$ of $Y$, the proper transform $S_Y$ of $S$ in $Y$ is defined by $$1+y^{p_1}z^{p_2}+y^{\frac{p_2}{\alpha}\beta+p_1}+x^{(\beta-1)m}z^{m}=0,$$ which is non-singular. 
Similarly, in the open chart $U_2\simeq \hat{\mathbb{C}}^3/\frac{1}{\alpha}(-1,1,-\beta)$ if $\alpha>1$ (resp. $U_2\simeq \hat{\mathbb{C}}^3$ if $\alpha=1$), $S_Y$ is defined by $$x^{m}+z^{p_2}+1+y^{(\beta-1)m}z^{m}=0,$$ which is non-singular.
In the open chart $U_3\simeq \hat{\mathbb{C}}^3/\frac{1}{\beta}(-1,-\alpha,1)$, $S_Y$ is defined by $$x^{m}+y^{p_1}+y^{\frac{p_2}{\alpha}\beta+p_1}+z^{(\beta-1)m}=0,$$ which is only singular at the origin if $p_1>0$ (resp. is nonsingular if $p_1=0$). 
If $p_1>0$, then the weighted multiplicity of $S_Y$ with respect to Kawamata blow up $\sigma_3':Z_3\to U_3$ is $$\beta\cdot v_3(x^{m}+y^{p_1}+y^{\frac{p_2}{\alpha}\beta+p_1}+z^{(\beta-1)m})=p_1.$$ By assumption that $p_2\geq p_1$, we have $$\frac{\alpha+\beta}{p_1\alpha+p_2\beta}\leq \frac{1}{p_1},$$
where $\frac{1}{p_1}$ is the canonical threshold of $(U_3,S_Y|_{U_3})$ near the origin of $U_3$ if $p_1>0$ by \cite{Ka}.
Note that $$\bar{s}+\bar{t}=\frac{1}{\beta}+\frac{\bar{t}}{\beta}(\alpha+\beta)\textup{ and } w_3(f)=p_1\bar{s}+p_2\bar{t}=\frac{p_1}{\beta}+\frac{\bar{t}}{\beta}(p_1\alpha+p_2\beta).$$
The above inequality $\frac{\alpha+\beta}{p_1\alpha+p_2\beta}\leq \frac{1}{p_1}$ gives $\frac{\alpha+\beta}{p_1\alpha+p_2\beta}\leq \frac{\bar{s}+\bar{t}}{w_3(f)}$.
It is then easy to see $\ct(\hat{\mathbb{C}}^3,S)=\min\{\frac{\alpha+\beta}{p_1\alpha+p_2\beta},1\}$ by Lemma \ref{compute smct}.  

\noindent 
\textbf{Case 2: $\alpha\nmid p_2$ and $p_1\neq p_2$}. 
Define $S$ to be the Cartier divisor $f=0$ in $\hat{\mathbb{C}}^3$ where $$f=x^{m}+y^{p_1}z^{p_2}+y^{\lfloor\frac{p_2}{\alpha}\rfloor \beta +p_1}z^{q}+y^{m}+z^{m}.$$ It's not hard to see that $S$ has only isolated singularities, say $Q_1,...,Q_n$. In particular, the divisor $S$ is prime. Without loss of generality, we assume that $Q_1$ is the origin of $X=\hat{\mathbb{C}}^3$. Note that $m=w(f)$ and $$ w_2(f)= w_2(y^{\lfloor\frac{p_2}{\alpha}\rfloor \beta +p_1}z^{q})=s(\lfloor\frac{p_2}{\alpha}\rfloor \beta +p_1)+tq \textup{ and }w_3(f)=w_3(y^{p_1}z^{p_2})=p_1\bar{s}+p_2\bar{t}.$$ 
As the above computation, outside the singular set $\{Q_2,...,Q_n\}$, the proper transform $S_Y$ of $S$ in $Y$ is non-singular in $U_1$ and is only singular at the origin $U_3$ (resp. $U_2$), where we have $$
\frac{\alpha+\beta}{p_1\alpha+p_2\beta}\leq \frac{\bar{s}+\bar{t}}{w_3(f)}.$$ In the open chart $U_2\simeq \hat{\mathbb{C}}^3/\frac{1}{\alpha}(-1,1,-\beta)$, $S_Y$ is defined by 
$$x^{m}+z^{p_2}+z^q+y^{(\alpha-1)m}+y^{(\beta-1)m}z^{m}=0,$$ which is only singular at the origin of $U_2$ and $Q_2,...,Q_n$.
Then the weighted multiplicity of $S_Y$ with respect to Kawamata blow up $\sigma_2'\colon Z_2\to U_2$ is $$\alpha\cdot v_2(x^{m}+z^{p_2}+z^q+y^{(\alpha-1)m}+y^{(\beta-1)m}z^{m})=\alpha\cdot v_2(z^q)=q,$$ where $v_2=\frac{1}{\alpha}(\alpha-s,s,1)$ and $q\in (0,\alpha-1]$ is an integer. 
Recall that $p_2-q=\alpha\lfloor \frac{p_2}{\alpha}\rfloor$ is a multiple of $\alpha$. As $p_1\neq p_2$, we have $p_2\geq \max\{\alpha,p_1\}$ by assumption. In particular, $p_2-q=\alpha\lfloor\frac{p_2}{\alpha}
 \rfloor\ge \alpha$ and 
$$ (p_2-q)\beta+p_1\alpha\geq \alpha \beta+p_1\alpha>q\alpha \textup{ and thus }\frac{\alpha+\beta}{p_1\alpha+p_2\beta}< \frac{1}{q},$$ 
where $\frac{1}{q}$ 
is the canonical threshold of $(U_2,S_Y|_{U_2})$ 
near the origin of $U_2$ 
by \cite{Ka}.
Note that $$s+t=\frac{1}{\alpha}+\frac{s}{\alpha}(\alpha+\beta)\textup{ and } w_2(f)=s(\lfloor\frac{p_2}{\alpha}\rfloor \beta +p_1)+tq=\frac{q}{\alpha}+\frac{s}{\alpha}(p_1\alpha+p_2\beta),$$
and the inequality $\frac{\alpha+\beta}{p_1\alpha+p_2\beta}< \frac{1}{q}$ implies $\frac{\alpha+\beta}{p_1\alpha+p_2\beta}< \frac{s+t}{w_2(f)}$.
It is then easy to see $\ct(\hat{\mathbb{C}}^3,S)=\frac{\alpha+\beta}{p_1\alpha+p_2\beta}$ by Lemma \ref{compute smct}. 

\noindent 
\textbf{Case 3: $\alpha\nmid p_2$ and $p_1= p_2$}. 
Define $S$ to be the Cartier divisor $f=0$ in $\hat{\mathbb{C}}^3$ where $$f=x^{m}+y^{p_2}z^{p_2}+y^{m}+z^{m}.$$ Similarly, $S$ is only singular at the origin and thus is prime. Note that   $$m=w(f),\ w_2(f)=w_2(y^{p_2}z^{p_2})=(s+t)p_2\text{ and }w_3(f)=w_3(y^{p_2}z^{p_2})=(\bar{s}+\bar{t})p_2.$$ 
Then the proper transform $S_Y$ is non-singular in the open chart $U_1$ and is only singular at the origin of $U_2$ (resp. $U_3$). As $p_1=p_2$, we have $\frac{\alpha+\beta}{p_1\alpha+p_2\beta}=\frac{1}{p_2}$ where $\frac{1}{p_2}$ is the canonical threshold of $(U_3,S')$ (resp. $(U_2,S')$) near the origin by \cite{Ka}. It is then easy to see $\ct(\hat{\mathbb{C}}^3,S)=\frac{\alpha+\beta}{p_2\alpha+p_2\beta}=\frac{1}{p_2}$ by Lemma \ref{compute smct}.
Thus, we have proved $\mathcal{T}^{\textup{can}}_{3,\textup{sm}}=C\cap [0,1]$. 

It remains to show $C\cap [0,1]=\mathcal{HT}_2$. Recall that $\mathcal{HT}_2$ consists of $1$ and $$\frac{c_1+c_2}{c_1c_2+a_1c_2+a_2c_1}$$ for some non-negative integers $a_1,a_2,c_1,c_2$ with $a_1+c_1\geq \max\{2,a_2\}$ and $a_2+c_2\geq \max\{2,a_1\}$ by \cite[(3.2)]{Kol08} (cf. \cite[Theorem 7.1]{K99a}).  

Suppose first that $1> \frac{\alpha+\beta}{p_1\alpha+p_2\beta}\in \mathcal{T}^{\textup{can}}_{3,\textup{sm}}$. 
If $\alpha+\beta$ divides $p_1\alpha+p_2\beta$, we put $c_1=c_2=1, a_1=\frac{p_1\alpha+p_2\beta}{\alpha+\beta}, a_2=a_1-1\geq 1$. One sees  $\frac{\alpha+\beta}{p_1\alpha+p_2\beta}=\frac{1+1}{1+a_1+(a_1-1)}\in \mathcal{HT}_2$. 
If $\alpha+\beta$ does not divide $p_1\alpha+p_2\beta$, one has $p_2>p_1$ and $p_2\geq \alpha$. Let $l$ be the non-negative integer with $$(\alpha+\beta)l<p_2-p_1< (\alpha+\beta)(l+1).$$ We put $c_1=\alpha, c_2=\beta, a_1=p_2-\alpha l-\alpha$ and $a_2=p_1+\beta l$. 
It is easy to see $a_1+c_1\geq \max\{2,a_2\}$ and $a_2+c_2\geq \max\{2,a_1\}$ and thus $\frac{\alpha+\beta}{p_1\alpha+p_2\beta}\in \mathcal{HT}_2$.

Conversely, suppose that $0\neq \frac{c_1+c_2}{c_1c_2+a_1c_2+a_2c_1}\in \mathcal{HT}_2$. If $c_i=0$ for some $i=1,2$, we put $p_1=p_2=a_i$. Then $\frac{c_1+c_2}{c_1c_2+a_1c_2+a_2c_1}=\frac{1}{a_i}\in \mathcal{T}^{\textup{can}}_{3,\textup{sm}}$. Thus we may assume that both $c_1$ and $c_2$ are positive. Let $d=\gcd(c_1,c_2)$ and $\alpha=\frac{c_1}{d}$ and $\beta=\frac{c_2}{d}$. We put $\alpha=c_1, \beta=c_2, p_1=a_2$ and $p_2=c_1+a_1$. Then we see $\frac{c_1+c_2}{c_1c_2+a_1c_2+a_2c_1}=\frac{\alpha+\beta}{p_1\alpha+p_2\beta}\in \mathcal{T}^{\textup{can}}_{3,\textup{sm}}$. This completes the proof of Theorem \ref{classfysmct}.

\end{proof}

\begin{rem}\label{quickremark1}
    Suppose that $\alpha$ and $\beta$ are two relative prime positive integers with $\alpha<\beta$. 
    If $m$ is an integer greater than $(\beta-1)\alpha+(\beta-2)\beta$, then $\frac{\alpha+\beta}{m}\in \mathcal{T}^{\textup{can}}_{3,\textup{sm}}$. Indeed, 
    as $(\beta-1)\alpha+(\beta-2)\beta\geq \alpha\beta-\alpha-\beta$, there exist non-negative integers $p_1$ and $p_2$ satisfying $m=p_1\alpha+p_2\beta$ and $p_1< \beta$. We have 
    $$
    p_2\beta=m-p_1\alpha>(\beta-1)\alpha+(\beta-2)\beta-p_1\alpha\geq \beta(\alpha+\beta-2)-\beta\alpha=(\beta-2)\beta,
    $$ whence $p_2\geq \beta-1\geq \max\{p_1,\alpha\}$. Therefore, $\frac{\alpha+\beta}{m}\in \mathcal{T}^{\textup{can}}_{3,\textup{sm}}$ by Theorem  \ref{classfysmct}. In particular, $\frac{3}{m} \in \mathcal{T}^{\textup{can}}_{3,\textup{sm}}$ (resp. $\frac{4}{m} \in \mathcal{T}^{\textup{can}}_{3,\textup{sm}}$) when $m\geq 3$ (resp. $m\geq 6$). It is also easy to see   $\frac{2}{m}=\ct(\hat{\bC}^3,S)\in \mathcal{T}^{\textup{can}}_{3,\textup{sm}}$ by considering $S:=\{ x^m+y^m+z^m=0\}$ in $\mathbb{C}^3$ (see \cite[Theorem 3.6]{Stepanov}) when $m\ge 2$. 
    Therefore, every number in $\aleph_4$ with the exception of $\frac{4}{5}$ is contained in $\mathcal{T}^{\textup{can}}_{3,\textup{sm}}(= C \cap [0,1])$. Recall that $\frac{4}{5}$ is a canonical threshold indicated in \cite[Example 3.11]{Prok08}. We can rewrite decomposition in (2) as

$$\mathcal{T}_{3}^{\textup{can}} = \{0\} \cup \{\frac{4}{5}\} \cup \mathcal{T}_{3,sm}^{\textup{can}} \cup \mathcal{T}_{3, cA,\geq 5}^{\textup{can}} \cup \mathcal{T}_{3, cA/m,\geq 5}^{\textup{can}} \cup \mathcal{T}_{3,cD,\geq 5}^{\textup{can}} \cup \mathcal{T}_{3, cD/2,\geq 5}^{\textup{can}} \eqno{(3).}$$

\end{rem}

\vspace{0.3cm}

\begin{section}{Classification of  $\mathcal{T}^{\textup{can}}_{3}$}
\end{section}
In this section, we shall establish the inclusion $\mathcal{T}_{3, *,\geq 5}^{\textup{can}}\subseteq C\cap [0,\frac{4}{5}]$ for each $*= cA, cA/n, cD$ and $cD/2$ where $$C=\left\{\frac{\alpha+\beta}{p_1\alpha+p_2\beta} \Bigg| \begin{array}{l} \alpha, \beta, p_2\in \mathbb{N} \textup{ and } p_1 \in \mathbb{Z}_{\ge 0 } \textup{ such that }\alpha\leq \beta,\ \\ \gcd(\alpha,\beta)=1 \textup{ and either } p_2\geq \max\{\alpha, p_1\} \textup{ or }p_2=p_1  \end{array} \right\}$$ in Propositions \ref{possible 3ctcA}, \ref{possible 3ctcA/n new}, \ref{classfycDct} and \ref{classfycD2ct}. The arguments are based on the idea of the inclusion $\mathcal{T}_{3, \textup{sm}}^{\textup{can}}\subseteq C\cap [0,1]$ in the proof of Theorem \ref{classfysmct} and reducing to coprime weights. Moreover, we will need to consider semi-invariant conditions in the non-Gorenstein cases in Propositions \ref{possible 3ctcA/n new} and \ref{classfycD2ct}.
\vspace{0.5cm}
\begin{prop} \label{possible 3ctcA}
We have $\mathcal{T}_{3, cA,\geq 5}^{\textup{can}}\subseteq C\cap [0,\frac{4}{5}]$.
\end{prop}
\begin{proof}
Let $\ct(X,S) \in \mathcal{T}_{3, cA}^{\textup{can}}$ be a canonical threshold. We have $\ct(X,S)\leq \frac{4}{5}$ by \cite{Prok08}. By \cite[(2.10) Proposition-definition]{Corti} and Theorem 1.2(i) in \cite{Kawakita05}, there exists an analytical identification $P\in X \simeq o\in (\varphi=xy+g(z,u)=0)$ in $\hat{\bC}^4$ where $o$ denotes the origin of $\hat{\bC}^4$ such that the canonical threshold $\ct(X,S)$ is computed by a weighted blow up $\sigma: Y \to X$ of weights $w=wt(x,y,z,u)=(r_{1},r_{2},a,1)$ satisfying the following:

\begin{itemize}
\item $w(g(z,u))=r_{1}+r_{2}=ad$ where $r_1,r_2,a,d$ are positive integers with $a\ge 5$;
\item $z^{d}\in g(z,u)$ and hence $w(g(z,u))=w(z^{d})$;
\item $\gcd(r_{1}, a)=\gcd(r_{2}, a)=1$;
\item $\ct(X,S)=\frac{a}{m}$ where $m=nw(f)$ and $S$ is defined by the formal power series $f=0$ analytically and locally.
\end{itemize}

We may assume $d\ge 2$ since otherwise $P\in X$ is nonsingular and is treated in the previous section.
Without loss of generality, we may assume $r_1\leq r_2$.
As $\gcd(a,r_1)=\gcd(a,r_2)=1$, there exist non-negative integers $s_i^*<r_i$ such that $1+a_ir_i=as_i^*$ for $i=1,2$. Note that $s_i^*=0$ if and only if $r_i=1$. From assumption that $a\ge 5$, we see $r_2\geq (r_1+r_2)/2=ad/2>1$.
Define 
 $$ w_2=(r_1-a_2d+s_2^*, r_2-s_2^*, a_1,1) \textup{ (resp. } w_1=(r_1-s_1^*, r_2-a_1d+s_1^*, a_2, 1) \textup{ if }r_1>1).$$
Since $w_1\succeq \frac{r_1-s_1^*}{r_1} w$ and $w_2\succeq \frac{r_2-s_2^*}{r_2} w$, by \cite[Lemmas 2.1 and 4.1]{3ct}, we see that 
$$ \lfloor \frac{a_1}{a}m \rfloor  \ge \lceil \frac{r_2-s_2^*}{r_2} m \rceil \textup{ (resp. } \lfloor \frac{a_2}{a}m \rfloor  \ge \lceil \frac{r_1-s_1^*}{r_1} m \rceil \textup{ if }r_1>1). $$
Denote by $h=\gcd(r_1,r_2)$. As the positive integer $r_i'=r_i/h$ has no common divisor with $a$, there exist non-negative integers $s_i^{*'}<r_i'$ and $a_i'<a$ such that $1+a_i'r_i'=as_i^{*'}$ for $i=1,2$. Again, it follows from the assumption $a\ge 5$ that $r_2'>1$. Note that $a_1+a_2=a$.
\begin{claim}\label{cA reduce coprime} 
$$\lfloor\frac{s_2^{*'}}{r_2'}m\rfloor\geq \lceil\frac{a_2'}{a}m\rceil \textup{ (resp. }\lfloor\frac{s_1^{*'}}{r_1'}m\rfloor\geq \lceil\frac{a_1'}{a}m\rceil \textup{ if }r_1>1).$$
\end{claim}
\begin{proof}[Proof of Claim \ref{cA reduce coprime}]
It follows from
   \begin{itemize}
   \item $r_2=hr_2'$, and 
   \item $1+a_2r_2=as_2^{*}$ and $0<a_2<a$ and $0<s_2^{*}<r_2$ and 
       \item $1+a_2'r_2'=as_2^{*'}$ and $0<a_2'<a$ and $0<s_2^{*'}<r_2'$
   \end{itemize}
   that there exists a non-negative integer $b_2$ with $s_2^*=b_2r_2'+s_2^{*'}$ and thus \[1+a_2r_2=as_2^*=ab_2r_2'+as_2^{*'}=ab_2r_2'+1+a_2'r_2'.\]
   This yields $a_2hr_2'=a_2r_2=r_2'(ab_2+a_2').$ In particular, $a_2'=a_2h-ab_2$. 
   
   From $\lfloor\frac{a_1}{a}m\rfloor\geq \lceil\frac{r_2-s_2^*}{r_2}m\rceil$ and $a=a_1+a_2$, one has $\lfloor\frac{s_2^*}{r_2}m\rfloor\geq \lceil\frac{a_2}{a}m\rceil$. Let $\eta_2$ be an integer with $\lfloor\frac{s_2^*}{r_2}m\rfloor\geq \eta_2\geq  \lceil\frac{a_2}{a}m\rceil$. 
  Then we see 
    \begin{align*}
   \frac{s_2^{*'}}{r_2'}m&=\frac{s_2^*-b_2r_2'}{r_2'}m=\frac{s_2^*}{r_2'}m-b_2m=\frac{s_2^*}{r_2}mh-b_2m\\
   &\ge \eta_2h-b_2m\geq \frac{a_2}{a}mh-b_2m=\frac{a_2'+ab_2}{a}m-b_2m=\frac{a_2'}{a}m.
   \end{align*}
   where $\eta_2h-b_2m$ is an integer. This gives the inequality $\lfloor\frac{s_2^{*'}}{r_2'}m\rfloor\geq \lceil\frac{a_2'}{a}m\rceil$.
   
   Similarly, the above argument shows that the inequality  $\lfloor\frac{a_2}{a}m\rfloor\geq \lceil\frac{r_1-s_1^*}{r_1}m\rceil$ implies $\lfloor\frac{s_1^{*'}}{r_1'}m\rfloor\geq \lceil\frac{a_1'}{a}m\rceil$ if $r_1>1$. If $r_1'=1$, then $a_1'=-1, s_1^{*'}=0$ and thus $\lfloor\frac{s_1^{*'}}{r_1'}m\rfloor=0\geq \lceil\frac{a_1'}{a}m\rceil$. We complete the proof of Claim \ref{cA reduce coprime}.
\end{proof}
 
From \cite[Proposition 4.2]{3ct}, we have $dm\ge r_1r_2$ if $r_1>1$. Note that $1\ge \ct=\frac{a}{m}=\frac{r_1+r_2}{dm}$, we have $dm\ge r_2=r_1r_2$ if $r_1=1$. Now $dm$ is an integer greater than $r_1r_2-r_1-r_2$, so there exist non-negative integers $p_1$ and $p_2$ with $dm=p_1r_1+p_2r_2$. Dividing $h=\gcd(r_1,r_2)$, we have $$d'm=p_1r_1'+p_2r_2' \textup{ where }d'=\frac{d}{h}.$$ Replacing $d'm=p_1r_1'+p_2r_2'$ by $d'm=(p_1-r_2'\lfloor \frac{p_1}{r_2'}\rfloor)r_1'+(p_2+\lfloor \frac{p_1}{r_2'}\rfloor r_1')r_2'$, one may assume that $0\le p_1<r_2'$. From $r_1'+r_2'=ad'$, we rewrite $d'm=p_1ad'+k_1r_2'$ where $k_1=p_2-p_1$ is an integer. We have $d'| k_1r_2'$. Since $\gcd(r_1',r_2')=1$ and $r_1'+r_2'=ad'$, we observe $\gcd(d',r_2')=1$ and hence $d'|k_1$. 

Now
    \begin{align*}
    \frac{a_2'}{a}m&=\frac{a_2'}{ad'}(p_1ad'+k_1r_2')=a_2'p_1+\frac{as_2^{*'}-1}{ad'}k_1=a_2'p_1+\frac{s_2^{*'}k_1}{d'}-\frac{k_1}{ad'},  \\
    \frac{s_2^{*'}}{r_2'}m&=\frac{as_2^{*'}}{d'r_2'}p_1d'+\frac{s_2^{*'}k_1}{d'}=\frac{1+a_2'r_2'}{r_2'}p_1+\frac{s_2^{*'}k_1}{d'}=a_2'p_1+\frac{p_1}{r_2'}+\frac{s_2^{*'}k_1}{d'}
    \end{align*}
     and $d'|k_1$, 
       one has $$\lceil\frac{a_2'}{a}m\rceil=a_2'p_1+\frac{s_2^{*'}k_1}{d'}+\lceil-\frac{k_1}{ad'}\rceil \textup{ and }\lfloor\frac{s_2^{*'}}{r_2'}m\rfloor=a_2'p_1+ \frac{s_2^{*'}k_1}{d'}+\lfloor\frac{p_1}{r_2'}\rfloor.$$
    From Claim \ref{cA reduce coprime}, we have $\lfloor\frac{s_2^{*'}}{r_2'}m\rfloor\ge \lceil\frac{a_2'}{a}m\rceil$. It is then the same to $\lfloor\frac{p_1}{r_2'}\rfloor\ge \lceil-\frac{k_1}{ad'}\rceil$. Since $p_1<r_2'$, we see $k_1\ge 0$ and hence $p_2\ge p_1$.
    
    Suppose that $r_1>1$.
    We rewrite $d'm=k_2r_1'+p_2ad'$ where $k_2=p_1-p_2=-k_1$ is a non-positive integer divisible by $d'$. 
    Note that 
    \begin{align*}
    \frac{a_1'}{a}m&=\frac{a_1'}{ad'}(k_2r_1'+p_2ad')=\frac{as_1^{*'}-1}{ad'}k_2+a_1'p_2=\frac{s_1^{*'}k_2}{d'}-\frac{k_2}{ad'}+a_1'p_2,\\
     \frac{s_1^{*'}}{r_1'}m&=\frac{s_1^{*'}k_2}{d'}+\frac{s_1^{*'}a}{d'r_1'}p_2d'=\frac{s_1^{*'}k_2}{d'}+\frac{1+a_1'r_1'}{d'r_1'}p_2d'=\frac{s_1^{*'}k_2}{d'}+a_1'p_2+\frac{p_2}{r_1'}
     \end{align*} and $d'|k_2$, so
        $$\lceil\frac{a_1'}{a}m\rceil=\frac{s_1^{*'}k_2}{d}+a_1'p_2+\lceil-\frac{k_2}{ad'}\rceil \textup{ and } \lfloor \frac{s_1^{*'}}{r_1'}m\rfloor=\frac{s_1^{*'}k_2}{d'}+a_1'p_2+\lceil \frac{p_2}{r_1'}\rceil.$$ 
    From Claim \ref{cA reduce coprime}, we have $\lfloor\frac{s_1^{*'}}{r_1'}m\rfloor\geq \lceil\frac{a_1'}{a}m\rceil$. This is the same with 
    $$\lfloor\frac{p_2}{r_1'}\rfloor\ge \lceil-\frac{k_2}{ad'}\rceil=\lceil \frac{p_2-p_1}{ad'}\rceil.$$
    From the discussion above that $p_2\ge p_1$, we conclude that either $p_1=p_2$ or $p_2\ge r_1'$ holds. 
    
    Suppose that $r_1=1$. As we have shown $p_2\ge p_1$ is an integer where $d'm=p_1r_1'+p_2r_2'$ is positive, we have $p_2\ge 1=r_1'.$ The proof is completed.  
\end{proof}

\begin{prop} \label{possible  3ctcA/n new}
We have $\mathcal{T}_{3, cA/n,\geq 5}^{\textup{can}}\subseteq C\cap [0,\frac{4}{5}]$.
\end{prop}
\begin{proof}
The argument is similar but more subtle than that of Proposition \ref{possible 3ctcA}.

Let $\ct(X,S)\in \mathcal{T}^{can}_{3,cA/n}$ be a canonical threshold. We have $\ct(X,S)\leq \frac{4}{5}$ by \cite{Prok08}.
By \cite[(2.10) Proposition-definition]{Corti} and Theorem 1.2(i) in \cite{Kawakita05}, there exist an analytical identification $$P\in X \simeq o\in (\varphi \colon xy+g(z^n,u)=0)\subset \bC^4/\frac{1}{n}(1,-1,b,0)$$ where $o$ denotes the origin of $\bC^4/ \frac{1}{n}(1,-1,b,0)$ and a weighted blow up $\sigma:Y\to X$ of weights $w=wt(x,y,z,u)=\frac{1}{n} (r_{1},r_{2},a,n)$ at the origin satisfying the following:
\begin{itemize}
\item $nw(\varphi)=r_1+r_2=adn$ where $r_1,r_2,a,d,n$ are positive integers with $a\ge 5$ and $n\ge 2$;
\item $z^{dn}\in g(z^n,u)$;
\item $a\equiv br_1$ (mod $n$) and $0<b<n$;
\item $\gcd(b,n)=\gcd(\frac{a-br_1}{n},r_1)=\gcd(\frac{a+br_2}{n},r_2)=1$. 
\item $\ct(X,S)=\frac{a}{m}$ where $m=nw(f)$ and $S$ is defined by the formal power series $f=0$ analytically and locally.
\end{itemize}

By interchanging $r_1$ and $r_2$, one may assume $r_1\leq r_2$. As $a\ge 5$ and $r_1+r_2=adn$, $r_2>1$. On the open subset $Y\cap \{\bar{y}\neq 0\}$, $Y$ is defined by $$\bar{x}+g(\bar{z}^n\bar{y}^{a},\bar{u}\bar{y})/\bar{y}^{ad}=0\subset \hat{\mathbb{C}}^4/\frac{1}{r_2}(-r_1,n,-a,-n),$$ which is isomorphic to $\hat{\mathbb{C}}^3/\frac{1}{r_2}(n,-a,-n)$. By terminal lemma, one has 
\begin{itemize}
    \item $\gcd(r_2,an)=1$ and hence $\gcd(r_1,an)=1$.
\end{itemize}
The conditions $\gcd(b,n)=\gcd(r_1,an)=1$ and $a\equiv br_1$ (mod $n$) imply  \begin{itemize}
    \item $\gcd(a,n)=1$.
\end{itemize}

In what follows, we construct two auxiliary weights $w_1$ and $w_2$ (cf. \cite[3.5]{CJK14}). Put $s_1:=\frac{a-br_1}{n}$ and $s_2:=\frac{a+br_2}{n}$. 
As $\gcd(r_i,s_i)=1$ for $i=1,2$, we  have:
$$\left\{ \begin{array}{cccc} a&=br_1+ns_1; \\  1&=q_1r_1+s^*_1s_1; \\ a&=-br_2+ns_2; \\ 1&=q_2r_2+s^*_2s_2 \end{array} \right.$$ for some integer $0\le s^*_i<r_i$ and some integer $q_i$. Denote by  
\[\delta_1:=-nq_1+bs^*_1,\ \ \delta_2:=-nq_2-bs^*_2.\]
Then we obtain the following useful identities 
\begin{itemize}
    \item $\delta_ir_i+n=as_i^*$, for $i=1,2$, with each $\delta_i\neq 0$ by \cite[Claims 1,2 in 3.5]{CJK14}.
\end{itemize}   
Since $5dn\leq adn=r_1+r_2$ and we have assumed $r_1\leq r_2$, one sees $\delta_2>0$.
Define 
\begin{align*}
&w_2=\frac{1}{n}(r_1-\delta_2 dn+s_2^*,r_2-s_2^*, a- \delta_2, n) \\
\textup{(resp. }&w_1=\frac{1}{n}(r_1-s_1^*, r_2-\delta_1 dn+s_1^*, a-\delta_1,n) \textup{ if }\delta_1>0).
\end{align*}
Note that $w_1\succeq \frac{r_1-s_1^*}{r_1} w$ and $w_2\succeq \frac{r_2-s_2^*}{r_2} w$. 
By \cite[Lemmas 2.1 and 5.1]{3ct}, we see that 
$$
\lfloor \frac{a-\delta_2}{a}m \rfloor  \ge m_2\ge \lceil \frac{r_2-s_2^*}{r_2} m \rceil \textup{ (resp.} \lfloor \frac{a-\delta_1}{a}m \rfloor  \ge m_1\ge \lceil \frac{r_1-s_1^*}{r_1} m \rceil \textup{ if }\delta_1>0), $$ where $m_2:=nw_2(\frak{m}_2)$ (resp. $m_1:=nw_1(\frak{m}_1)$) is a weighted multiplicity for some monomial $\frak{m}_2\in f$ (resp. $\frak{m}_1\in f$ if $\delta_1>0$). Note that 
\begin{align*}
    m_2&=nw_2(\frak{m}_2)\equiv  (r_2-s_2^*)r_2^{-1} nw(\frak{m}_2)\equiv (r_2-s_2^*)r_2^{-1}nw(f)\\
    &\equiv (r_2-s_2^*)r_2^{-1}m \equiv (a-\delta_2)a^{-1}m \textup{ (mod }n),\end{align*} where $a^{-1}$ (resp. $r_2^{-1}$) denotes the inverse of $a$ (resp. $r_2$) modulo $n$. Similarly, $m_1\equiv (r_1-s_1^*)r_1^{-1}m$ (mod $n$) if $\delta_1>0$.

\begin{claim}\label{m >= r2}
Suppose that $a\nmid m$. Then $m\geq r_2$.
\end{claim}
\begin{proof}[Proof of Claim \ref{m >= r2}]
Let $\xi:=m-m_2$ where $m_2$ is an integer above with the properties that $$ \lfloor \frac{a-\delta_2}{a}m \rfloor \ge m_2 \ge \lceil \frac{r_2-s_2^*}{r_2} m \rceil   \textup{ \ \ and } m_2\equiv (a-\delta_2) a^{-1}m \textup{ (mod $n$)}.$$ We have $$ \lfloor \frac{s_2^*}{r_2} m \rfloor \ge \xi \ge \lceil \frac{\delta_2}{a}m  \rceil   \textup{ \ \ and } \xi\equiv \delta_2 a^{-1}m \textup{ (mod $n$)}.$$

Suppose on the contrary that $a\nmid m$ and $m<r_2$. Denote by $t_1$ and $t_2$ the positive integers with $m=r_2t_1-at_2$ and $a-1\geq t_1\ge 1$. Note that $t_1$ is the smallest positive integer with $t_1\equiv mr_2^{-1}$ (mod $a$) where $r_2^{-1}$ denotes the inverse of $r_2$ modulo $a$. From the equation $r_2\delta_2+n=as_2^*$, one sees 
$$ \frac{\delta_2}{a}m=\frac{\delta_2}{a}(r_2t_1-at_2)=\frac{(as_2^*-n)t_1}{a}-\delta_2t_2=s_2^*t_1-\delta_2t_2-\frac{nt_1}{a},$$
which implies 
$$\xi\equiv \delta_2 a^{-1}m\equiv a^{-1}(as_2^*t_1-a\delta_2t_2-nt_1)\equiv s_2^*t_1-\delta_2t_2 \textup{ (mod $n$)}.$$
As $n>\frac{nt_1}{a}>0$ and $\xi\geq  \lceil \frac{\delta_2}{a}m  \rceil$ where $\xi\equiv \delta_2 a^{-1}m \textup{ (mod $n$)}$, this gives $\xi\geq s_2^*t_1-\delta_2t_2$.
On the other hand, we have 
\begin{align*}
\frac{s_2^*}{r_2}m&=\frac{s_2^*}{r_2}(r_2t_1-at_2)=s_2^*t_1-\frac{(r_2\delta_2+n)t_2}{r_2}=s_2^*t_1-\delta_2t_2-\frac{nt_2}{r_2}\\
&\leq \xi-\frac{nt_2}{r_2}<\xi,
\end{align*}
where the last inequality holds as $r_2t_1-at_2=m<r_2$. This leads to a contradiction that $\lfloor \frac{s_2^*}{r_2} m \rfloor \ge \xi>\frac{s_2^*}{r_2}m\geq \lfloor \frac{s_2^*}{r_2} m \rfloor$. The proof of Claim \ref{m >= r2} is finished.
\end{proof}

\begin{claim}\label{cAnub} Suppose that 
$a\nmid m$.
Then $m \ge \frac{r_1 r_2}{dn}$.
\end{claim}
\begin{proof}[Proof of Claim \ref{cAnub}]
Suppose that $\delta_1<0$. Then $r_1\leq (-\delta_1)r_1+as_1^*=n\leq dn$ and hence $dnm\geq r_1m\geq r_1r_2$ by Claim \ref{m >= r2}.

We may thus assume that $\delta_1>0$. It follows from \cite[Remark 3.3]{CJK14} that $\delta_1+\delta_2=a$.
Since $\gcd(b,n)=1$ and $a\nmid m$ where $m$ is an integral combination of $r_1,r_2,a,n$,
one has that $a\nmid \delta_1m$. One sees that
$$\lfloor \frac{a-\delta_1}{a}m \rfloor +  \lfloor \frac{a-\delta_2}{a}m \rfloor=\lfloor \frac{a-\delta_1}{a}m \rfloor +  \lfloor \frac{\delta_1}{a}m \rfloor=m-1.$$
Recall that $m_i\equiv (r_i-s_i^*)r_i^{-1}m\equiv (a-\delta_i)a^{-1}m$ (mod $n$) for $i=1,2$. Therefore 
$$m_1+m_2\equiv (2a-(\delta_1+\delta_2))a^{-1}m\equiv aa^{-1}m\equiv m \textup{ (mod $n$)}.$$
Together with $m-1= \lfloor \frac{a-\delta_1}{a}m \rfloor +  \lfloor \frac{a-\delta_2}{a}m \rfloor \geq m_1+m_2$, one observes $$m-n\geq m_1+m_2\geq \lceil \frac{r_1-s_1^*}{r_1} m \rceil+ \lceil \frac{r_2-s_2^*}{r_2} m \rceil.$$
Note that $$r_1s_2^*+r_2s_1^*=\frac{a(r_1s_2^*+r_2s_1^*)}{a}=\frac{r_1(r_2\delta_2+n)+r_2(r_1\delta_1+n)}{a}=r_1r_2+dn^2.$$

Suppose on the contrary that $m < \frac{r_1r_2}{dn}$. Then we obtain 
\begin{align*} &\lceil \frac{r_1-s_1^*}{r_1} m \rceil+\lceil \frac{r_2-s_2^*}{r_2} m \rceil  \ge  \lceil \frac{r_1-s_1^*}{r_1} m + \frac{r_2-s_2^*}{r_2} m \rceil \\&= \lceil 2m-\frac{r_1s_2^*+r_2s_1^*}{r_1r_2}m\rceil = \lceil m-\frac{dn^2}{r_1r_2}m \rceil=m-\lfloor \frac{dnm}{r_1r_2}n\rfloor\geq m-n+1
\end{align*}
which contradicts to $m-n\geq \lceil \frac{r_1-s_1^*}{r_1} m \rceil+\lceil \frac{r_2-s_2^*}{r_2} m \rceil$. 
The proof of Claim \ref{cAnub} is finished.
\end{proof}

Denote by $h:=\gcd(r_1,r_2)$ and $b'$ the smallest positive integer with $b'\equiv bh$ (mod $n$). Let $d'$ and  $r_i'$ be positive integers with $d=d'h$ and $r_i=r_i'h$ for $i=1,2$. Put $s_1'=\frac{a-b'r_1'}{n}$ and $s_2'=\frac{a+b'r_2'}{n}$.
Since the integer $a$ is relatively prime to $r_i=hr_i'$, we have the following:
$$\left\{ \begin{array}{llll} a&=b'r_1'+ns_1'; \\  1&=q_1'r_1'+s^{*'}_1s_1'; \\ a&=-b'r_2'+ns_2'; \\ 1&=q_2'r_2'+s^{*'}_2s_2 \end{array} \right.$$ for some integer $0\le s^{*'}_i<r_i'$ and some integer $q_i'$. Let 
\[\delta_i':=-nq_i'+b's^{*'}_i, \textup{ for }i=1,2.\] 
As above, we observe the following 
\begin{itemize}
\item $\delta_i'r_i'+n=as_i^{*'}$ and $0\neq \delta_i'<a$ for $i=1,2$.
\item $\delta_2'>0$ and $a|\delta_1'+\delta_2'$.
\item if $\delta_1'>0$, then $a=\delta_1'+\delta_2'$.
\end{itemize}

\begin{claim}\label{cAn reduce coprime} There exists an integer $\xi_2$ with 
$$ \lfloor\frac{s_2^{*'}}{r_2'}m\rfloor\geq \xi_2\geq \lceil\frac{\delta_2'}{a}m\rceil \textup{ and }\xi_2\equiv s_2^{*'}({r_2'}^{-1})m\ \textup{(mod }n).$$ 
Similarly, in the case $\delta_1>0$, there exists an integer $\xi_1$ with
$$ \lfloor\frac{s_1^{*'}}{r_1'}m\rfloor\geq \xi_1\geq \lceil\frac{\delta_1'}{a}m\rceil \textup{ and } \xi_1\equiv s_1^{*'}({r_1'}^{-1})m\ \textup{(mod }n\textup{)}.$$
\end{claim}
\begin{proof}[Proof of Claim \ref{cAn reduce coprime}]
It follows from
   \begin{itemize}
   \item $r_2=hr_2'$ and $\gcd(a,r_2)=1$ and 
   \item $\delta_2r_2+n=as_2^{*}$  and $0<s_2^{*}<r_2$ and 
       \item $\delta_2'r_2'+n=as_2^{*'}$ and $0<s_2^{*'}<r_2'$
   \end{itemize}
   that there exists an integer $b_2$ with $s_2^*=b_2r_2'+s_2^{*'}$ and thus \[\delta_2r_2+n=as_2^*=ab_2r_2'+as_2^{*'}=ab_2r_2'+\delta_2'r_2'+n.\]
   This yields $\delta_2hr_2'=\delta_2r_2=r_2'(ab_2+\delta_2').$ In particular, $\delta_2h=ab_2+\delta_2'$. 
   
   Recall that $\lfloor\frac{a-\delta_2}{a}m\rfloor\geq m_2\geq \lceil\frac{r_2-s_2^*}{r_2}m\rceil$ where $m_2\equiv (r_2-s_2^*)r_2^{-1}m$ (mod $n$). 
   Then we see 
   \begin{align*}
   \frac{s_2^{*'}}{r_2'}m&=\frac{s_2^*-b_2r_2'}{r_2'}m=\frac{s_2^*}{r_2'}m-b_2m=\frac{s_2^*}{r_2}mh-b_2m\\
   &\ge (m-m_2)h-b_2m\geq \frac{\delta_2}{a}mh-b_2m=\frac{\delta_2'+ab_2}{a}m-b_2m=\frac{\delta_2'}{a}m.
   \end{align*}
   Denote by $\xi_2=(m-m_2)h-b_2m$. Then 
   $$  \xi_2 = (m-m_2)h-b_2m \equiv s_2^*r_2^{-1}mh-b_2m\equiv s_2^*{r_2'}^{-1}m-b_2m \equiv s_2^{*'}{r_2'}^{-1}m \textup{ (mod } n\textup{)}.$$
      
      Suppose that $\delta_1>0$. 
      The above argument works by interchanging indices $1$ and 2. We complete the proof of Claim \ref{cAn reduce coprime}.
\end{proof}

From Claim \ref{cAnub}, we have $dmn\ge r_1r_2$. So there exist non-negative integers $p_1$ and $p_2$ with $dmn=p_1r_1+p_2r_2$. Dividing $h=\gcd(r_1,r_2)$, we have $$d'mn=p_1r_1'+p_2r_2'.$$ Replacing by $d'mn=(p_1-r_2'\lfloor \frac{p_1}{r_2'}\rfloor)r_1'+(p_2+\lfloor \frac{p_1}{r_2'}\rfloor r_1')r_2'$, one may assume that $0\le p_1<r_2'$. From $r_1'+r_2'=ad'n$, we rewrite $d'mn=p_1ad'n+k_1r_2'$ where $k_1=p_2-p_1$ is an integer.
Note that 
    \begin{align*}
        \frac{\delta_2'}{a}m&=\frac{\delta_2'}{ad'n}(p_1ad'n+k_1r_2')=\delta_2'p_1+\frac{r_2'\delta_2'k_1}{ad'n}=\delta_2'p_1+\frac{k_1(as_2^{*'}-n)}{ad'n}\\
        &=\delta_2'p_1+\frac{s_2^{*'}k_1}{d'n}-\frac{k_1}{ad'} 
        \end{align*}
    \begin{align*}\textup{ and } \frac{s_2^{*'}}{r_2'}m&=\frac{s_2^{*'}}{d'nr_2'}(p_1ad'n+k_1r_2')=\frac{as_2^{*'}p_1}{r_2'}+\frac{s_2^{*'}k_1}{d'n}=\frac{(\delta_2'r_2'+n)p_1}{r_2'}+\frac{s_2^{*'}k_1}{d'n}\\&=\delta_2'p_1+\frac{s_2^{*'}k_1}{d'n}+\frac{np_1}{r_2'}.
    \end{align*} 
       \begin{claim} \label{int1}$d'n|k_1$.
    \end{claim}
    \begin{proof}[Proof of Claim \ref{int1}]  As $d'mn=p_1ad'n+k_1r_2'$, we have $d'n| k_1r_2'$. Since $\gcd(r_1',r_2')=1$ and $r_1'+r_2'=ad'n$, it follows that $\gcd(d'n,r_2')=1$ and hence $d'n|k_1$ and the claim is verified.
    \end{proof}
    
    From Claims \ref{cAn reduce coprime} and \ref{int1}, there exists an integer $\xi_2$ with $$\xi_2\equiv s_2^{*'}({r_2'}^{-1})m\equiv {r_2'}^{-1}(r_2'(\delta_2'p_1+\frac{s_2^{*'}k_1}{d'n}+\frac{np_1}{r_2'}))\equiv \delta_2'p_1+\frac{s_2^{*'}k_1}{d'n}\  \textup{(mod }n).$$ Let $l_2$ be the integer with $\xi_2=\delta_2'p_1+\frac{s_2^{*'}k_1}{d'n}+nl_2$. Then the inequalities $\lfloor\frac{s_2^{*'}}{r_2'}m\rfloor\ge \xi_2\geq \lceil\frac{\delta_2'}{a}m\rceil$ in Claim \ref{cAn reduce coprime} is then the same to $\lfloor\frac{np_1}{r_2'}\rfloor\ge nl_2\geq \lceil-\frac{k_1}{ad'}\rceil$. As $p_1<r_2'$, we see $k_1\ge 0$ and hence $p_2\ge p_1$.

      Suppose that $\delta_1<0$. Then $r_1\leq -\delta_1r_1=n-as_1^*\leq n$. If $p_2\geq n$, then we get the desired inequalities  $p_2\geq n\geq r_1\geq r_1'$. We shall rule out the case  $p_1<p_2\leq n-1$. By Claim \ref{m >= r2}, we have $$dnr_2\leq dmn=p_1r_1+p_2r_2\leq (n-1)(r_1+r_2).$$ So $((d-1)n+1)r_2\leq (n-1)r_1$. As we have assumed $r_1\leq r_2$, one has $d=1$. Thus, $d'=1$ and $n|k_1=p_2-p_1$ by Claim \ref{int1}. However, this leads to a contradiction $n\leq k_1=p_2-p_1\leq n-1-p_1\leq n-1$ provided that $p_1<p_2\leq n-1$.
     
     Suppose that $\delta_1>0$. Write $d'nm=k_2r_1'+p_2ad'n$ where $k_2=-k_1=p_1-p_2$.
    Note that 
    \begin{align*}
        \frac{\delta_1'}{a}m&=\frac{\delta_1'}{ad'n}(p_2ad'n+k_2r_1')=\delta_1'p_2+\frac{r_1'\delta_1'k_2}{ad'n}=\delta_1'p_2+\frac{(as_1^{*'}-n)k_2}{ad'n}\\
        &=\delta_1'p_2+\frac{s_1^{*'}k_2}{d'n}-\frac{k_2}{ad'} \\
    \textup{ and } \frac{s_1^{*'}}{r_1'}m&=\frac{s_1^{*'}}{d'nr_1'}(p_2ad'n+k_2r_1')=\frac{as_1^{*'}p_2}{r_1'}+\frac{s_1^{*'}k_2}{d'n}=\frac{(\delta_1'r_1'+n)p_2}{r_1'}+\frac{s_1^{*'}k_2}{d'n}\\&=\delta_1'p_2+\frac{s_1^{*'}k_2}{d'n}+\frac{np_2}{r_1'}.
    \end{align*} 
    From Claims \ref{cAn reduce coprime} and \ref{int1}, there exists an integer $\xi_1$ with $$\xi_1\equiv s_1^{*'}({r_1'}^{-1})m\equiv {r_1'}^{-1}(r_1'(\delta_1'p_2+\frac{s_1^{*'}k_2}{d'n}+\frac{np_2}{r_1'}))\equiv \delta_1'p_2+\frac{s_1^{*'}k_2}{d'n}  \textup{(mod }n).$$ Let $l_1$ be the integer with $\xi_1=\delta_1'p_2+\frac{s_1^{*'}k_2}{d'n}+nl_1$. Then the inequalities $\lfloor\frac{s_1^{*'}}{r_1'}m\rfloor\ge \xi_1\geq \lceil\frac{\delta_1'}{a}m\rceil$ in Claim \ref{cAn reduce coprime} are then the same to $\lfloor\frac{np_2}{r_1'}\rfloor\ge nl_1\geq \lceil-\frac{k_2}{ad'}\rceil$. As $-k_2=k_1\ge 0$, we see $np_2/r_1' \ge n$ and  hence $p_2\ge r_1'$ provided that $p_2>p_1$.   
We complete the argument of Proposition \ref{possible 3ctcA/n new}.
\end{proof}

\begin{prop}\label{classfycDct}
    we have  $\mathcal{T}^{\textup{can}}_{3,cD,\ge 5}\subseteq C\cap [0,\frac{4}{5}]$.
\end{prop}
\begin{proof}


Given a canonical threshold $\ct(X,S)\in \mathcal{T}_{3,cD,\ge 5}^{\textup{can}}$. By \cite{Prok08}, we see $\ct(X,S)\leq \frac{4}{5}$. By \cite[(2.10) Proposition-definition]{Corti} and the classification of Kawakita \cite[Theorem 1.2]{Kawakita05}, one sees that $\ct(X,S)$ is realized by a divisorial contraction $\sigma: Y \to X$ classified in Case 1 and Case 2.

\noindent {\bf Case 1.}
There exists an analytical identification: \[ (P\in X)\simeq  o\in ( \varphi: x^2+xq(z,u)+y^2u+\lambda y z^2+\mu z^3+p(y,z,u)=0) \subset \hat{\mathbb{C}}^4,\]
where $o$ denotes the origin of $\hat{\mathbb{C}}^4$ 
such that $\sigma: Y \to X$  is a weighted blow up of weights $w=wt(x,y,z,u)=(r+1,r,a,1)$ with center $P\in X$ and 
\begin{itemize}
\item $2r+1=ad$ where $d\ge 3$ and the integer $a\ge 5$ is odd,
\item $\ct(X,S)=\frac{a}{m}$ where $m=w(f)$ and $S$ is defined by the formal power series $f=0$ analytically and locally.
\end{itemize}


Let $\sigma_1\colon Y_1\to X$ (resp. $\sigma_2\colon Y_2\to X$) be the weighted blow up with weights $w_1=(r+1-d,r-d,a-2,1)$ (resp.  $w_2=(d,d,2,1)$) at the origin $P\in X$. By \cite[Lemma 6.3]{3ct} (see also \cite[Case Ic]{CJK15}), 
the exceptional set of $\sigma_1$ is a prime divisor. From the computation in \cite[Claim 6.6]{3ct}, the defining equation of the exceptional set of $\sigma_2$ is $x^2+\eta z^d$ with odd integer $d\geq 3$ for some nonzero constant $\eta$ and hence the exceptional set of $\sigma_2$ is a prime divisor.
As $2r+1=ad$, it yields
\[w_1\succeq \frac{r-d}{r} w\ \ \textup{ and } \ \ \ w_2\succeq \frac{d}{r+1} w.\]
It follows from \cite[Lemma 2.1]{3ct} that \[ \lfloor \frac{a-2}{a}m\rfloor \ge m_1\ge \lceil \frac{r-d}{r}m\rceil \ \ \textup{ and } \ \ \ \lfloor \frac{2}{a}m\rfloor \ge m_2\ge \lceil \frac{d}{r+1} m\rceil \eqno{\dagger_1} \]
where $m_1:=w_1(f)$ and $m_2:=w_2(f)$ denotes the corresponding weighted multiplicities.
\begin{claim}\label{r1bound}\footnote{Claims \ref{r1bound} and \ref{r2bound} were first obtained in \cite{HLL}.} If $a\nmid m$, then $dm\geq r(r+1)$. 
\end{claim}
\begin{proof}[Proof of Claim \ref{r1bound}]
Suppose on the contrary that $dm<(r+1)r$. We have
\begin{align*}
&  \lceil \frac{r-d}{r}m\rceil+\lceil \frac{d}{r+1}m\rceil \ge   \lceil \frac{r-d}{r}m+\frac{d}{r+1}m \rceil  = \lceil m-\frac{dm}{r(r+1)} \rceil=m.
\end{align*}
However, $a$ is odd and $a\nmid m$, hence  $\frac{2m}{a}$ is not an integer.
This implies 
\[\lfloor \frac{a-2}{a}m\rfloor+\lfloor \frac{2}{a}m\rfloor =m-1, \]
which contradicts to $\dagger_1$. This verifies Claim \ref{r1bound}.
\end{proof}

From Claim \ref{r1bound}, express $dm=p_1r+p_2(r+1)$ for some non-negative integers $p_1$ and $p_2$ with $p_1<r+1$. Note that \begin{align*}
    \frac{2}{a}m&=\frac{2dm}{ad}=\frac{2p_1r+2p_2(r+1)}{2r+1}=p_1+p_2+\frac{p_2-p_1}{2r+1}, \\
    \frac{d}{r+1}m&=\frac{p_1r+p_2(r+1)}{r+1}=p_1+p_2-\frac{p_1}{r+1},\\
    \frac{d}{r}m&=\frac{p_1r+p_2(r+1)}{r}=p_1+p_2+\frac{p_2}{r}.
\end{align*} 
As $p_1$ and $p_2$ are integers, 
the inequality $\lfloor \frac{2}{a}m\rfloor \ge \lceil \frac{d}{r+1} m\rceil$ then implies $\lfloor \frac{p_2-p_1}{2r+1} \rfloor\geq \lceil -\frac{p_1}{r+1}\rceil=-\lfloor \frac{p_1}{r+1}\rfloor=0$. In particular, $p_2\ge p_1$. Similarly, the inequality $\lfloor \frac{a-2}{a}m\rfloor \ge \lceil \frac{r-d}{r}m\rceil$ gives $$p_1+p_2+\lfloor\frac{p_2}{r}\rfloor=\lfloor\frac{d}{r}m \rfloor\geq \lceil \frac{2}{a}m\rceil=p_1+p_2+\lceil\frac{p_2-p_1}{2r+1}\rceil.$$ 
In particular, $p_2\ge r$ provided that $p_2-p_1>0$.

\noindent {\bf Case 2.} 
There exists an analytical identification: $$P\in X\simeq o\in \left( \begin{array}{ll}
\varphi_{1} \colon x^2+yt+p(y,z,u)=0 ;\\
 \varphi_{2} \colon  yu+z^{d}+q(z,u)u+t=0 \\
\end{array} \right) \subset \hat{\mathbb{C}}^5
$$
where $o$ denotes the origin of $\hat{\mathbb{C}}^5$ such that $\sigma: Y \to X$  is a weighted blow up of weights $w=(r+1,r,a,1,r+2)$ with center $P\in X$ and 
\begin{itemize}
\item $r+1=ad$ where $d\ge 2$ and $a\ge 5$,
\item $\ct(X,S)=\frac{a}{m}$ where $m=w(f)$ and $S$ is defined by the formal power series $f=0$ analytically and locally.
\end{itemize}


Compare the weights $w$ with the weights $w_1=(r-d+1,r-d,a-1,1,r-d+2)$ and $w_{2}=(d,d,1,1,d)$. 
By \cite[Lemma 2.1, Lemma 6.7 and Lemma 6.8]{3ct}, we have
$$ \lfloor \frac{a-1}{a} m \rfloor  \ge \lceil \frac{r-d}{r} m \rceil \textup{ and } \lfloor \frac{1}{a} m \rfloor  \ge \lceil \frac{d}{r+2} m \rceil.   \eqno{\dagger_2}$$

\begin{claim}\label{r2bound} If $a\nmid m$, then $2dm\geq (r+2)r$. 
\end{claim}
\begin{proof}[Proof of Claim \ref{r2bound}]
Suppose on the contrary that $2dm< (r+2)r$. We have 
\begin{align*}
&  \lceil \frac{r-d}{r}m\rceil+\lceil \frac{d}{r+2}m\rceil \ge   \lceil  \frac{r-d}{r}m+\frac{d}{r+2}m\rceil  = \lceil m-\frac{2dm}{r(r+2)} \rceil=m.
\end{align*}
However, $a\nmid m$, hence  
\[\lfloor \frac{a-1}{a}m\rfloor+\lfloor \frac{1}{a}m\rfloor =m-1, \]
which contradicts to $\dagger_2$. The proof of Claim \ref{r2bound} is complete.
\end{proof}
From Claim \ref{r2bound}, express $2dm=p_1r+p_2(r+2)$ for some non-negative integers $p_1$ and $p_2$.
Denote by  $h:=\gcd(r,r+2)$.
We have $$\frac{2dm}{h}=p_1\frac{r}{h}+p_2\frac{r+2}{h}.$$ By replacing $p_2$  by $p_2+\lfloor\frac{p_1h}{r+2}\rfloor\frac{r}{h}$ (resp. replacing $p_1$ by $p_1-\lfloor\frac{p_1h}{r+2}\rfloor \frac{r+2}{h}$), we may assume that $0\le p_1<\frac{r+2}{h}$. Note that \begin{align*}
    \frac{1}{a}m&=\frac{2dm}{2ad}=\frac{p_1r+p_2(r+2)}{2(r+1)}=\frac{p_1+p_2}{2}+\frac{p_2-p_1}{2(r+1)}, \\
    \frac{d}{r+2}m&=\frac{2dm}{2(r+2)}=\frac{p_1r+p_2(r+2)}{2(r+2)}=\frac{p_1+p_2}{2}-\frac{p_1}{r+2},\\
    \frac{d}{r}m&=\frac{2dm}{2r}=\frac{p_1r+p_2(r+2)}{2r}=\frac{p_1+p_2}{2}+\frac{p_2}{r}.
\end{align*} 

Suppose that the integer $p_1+p_2$ is even.
The inequality $\lfloor \frac{1}{a} m \rfloor  \ge \lceil \frac{d}{r+2} m \rceil$ then implies $\lfloor \frac{p_2-p_1}{2(r+1)} \rfloor\geq \lceil -\frac{p_1}{r+2}\rceil=0$. In particular, $p_2\ge p_1$. Similarly, the inequality $\lfloor \frac{a-1}{a}m\rfloor \geq \lceil \frac{r-d}{r}m\rceil$ gives $\lfloor\frac{p_2}{r}\rfloor\ge \lceil\frac{p_2-p_1}{2(r+1)}\rceil.$ In particular, $p_2\ge r$ provided that $p_2-p_1>0$.

Suppose next that the integer $p_1+p_2$ is odd. Since $(p_1+p_2)r=2(dm-p_2)$, $r$ is even and $h=2$. 
The inequality $\lfloor \frac{1}{a} m \rfloor  \ge \lceil \frac{d}{r+2} m \rceil$ implies $\lfloor \frac{1}{2}+\frac{p_2-p_1}{2(r+1)} \rfloor\geq \lceil \frac{1}{2}-\frac{p_1}{r+2}\rceil=1$ as $p_1<\frac{r+2}{h}=\frac{r+2}{2}$. In particular, $p_2\ge p_1+r+1$. Similarly, the inequality $\lfloor \frac{a-1}{a}m\rfloor \geq \lceil \frac{r-d}{r}m\rceil$ yields $\lfloor\frac{1}{2}+\frac{p_2}{r}\rfloor\geq \lceil\frac{1}{2}+\frac{p_2-p_1}{2(r+1)}\rceil\ge 1.$ In particular, $p_2\ge \frac{r}{2}$. 
The proof of Proposition \ref{classfycDct} is complete.   
\end{proof}

\begin{prop}\label{classfycD2ct}
    We have  $\mathcal{T}^{\textup{can}}_{3,cD/2,\ge 5}\subseteq C\cap [0,\frac{4}{5}]$. 
\end{prop}
\begin{proof}
Given a canonical threshold $\ct(X,S)\in \mathcal{T}_{3,cD/2}^{\textup{can}}$.
We have $\ct(X,S)\leq \frac{4}{5}$ by \cite{Prok08}. By \cite[(2.10) Proposition-definition]{Corti} and the classification of Kawakita \cite[Theorem 1.2]{Kawakita05}, one sees that $\ct(X,S)$ is realized by a divisorial contraction $\sigma: Y \to X$ given in Case 1 and Case 2.

\noindent {\bf Case 1.}
There exists an analytical identification: \[ P\in X\simeq o\in (\varphi: x^2+xzq(z^2,u)+y^2u+\lambda y z^{2\alpha -1}+p(z^2,u)=0 ) \subset \hat{\mathbb{C}}^4/\frac{1}{2}(1,1,1,0)\]
where $o$ denotes the origin of $\hat{\mathbb{C}}^4/\frac{1}{2}(1,1,1,0)$ such that $\sigma: Y \to X$  is a weighted blow up of weights $w=\frac{1}{2}(r+2,r,a,2)$ with center $P\in X$ and 
\begin{itemize}
\item $r+1=ad$ where both $a\geq 5$ and $r$ are odd;
\item $\ct(X,S)=\frac{a}{m}$ where $m=w(f)$ and $S$ is defined by the formal power series $f=0$ analytically and locally.
\end{itemize}
Denote by $\sigma_1\colon Y_1\to X$ and $\sigma_2\colon Y_2\to X$ the weighted blow ups with weights $w_1=\frac{1}{2}(r+2-2d,r-2d,a-2,2)$ and $w_2=\frac{1}{2}(2d,2d,2,2)$ at the origin $P\in X$ and $m_1:=2w_1(f)$ and $m_2:=2w_2(f)$ the weighted multiplicities respectively. By \cite[Lemma 7.3]{3ct}, the exceptional set of $\sigma_1$ is a prime divisor. Note that the exceptional set of $\sigma_2$ is a $\mathbb{Z}_2$ quotient of $$\{x^2+z^{2d}=0\}\subset \mathbb{P}(2d,2d,2,2).$$ It is a prime divisor. 
By \cite[Lemma 2.1]{3ct} and $r+1=ad$, we have
$$ \lfloor\frac{a-2}{a} m \rfloor  \ge m_1\ge \lceil \frac{r-2d}{r} m \rceil \textup{ and }\lfloor\frac{2}{a} m \rfloor  \ge m_2\ge \lceil \frac{2d}{r+2} m \rceil $$ with $m_1\equiv (a-2)a^{-1}m\equiv m$ and $m_2\equiv 2a^{-1}m\equiv 0$ (mod $2$). See Remark \ref{alternative} for an alternative explanation of the inequality $\lfloor\frac{2}{a} m \rfloor  \ge m_2$.

\begin{claim} \label{raboundcD21} If $a\nmid m$, then $2dm\geq r(r+2)$. 
\end{claim} 
\begin{proof}[Proof of Claim \ref{raboundcD21}]
Since $a\nmid m$ and $a$ is odd, we have 
$m-1=\lfloor \frac{a-2}{a}m\rfloor+\lfloor \frac{2m}{a}\rfloor \geq m_1+m_2$. As $m_1+m_2\equiv m$ (mod 2), one has $m-2\geq m_1+m_2$.

Suppose on the contrary that $2dm< r(r+2)$.
We see 
\begin{align*}
m_1+m_2&\geq \lceil \frac{r-2d}{r}m\rceil + \lceil \frac{2d}{r+2} m\rceil\ge   \lceil \frac{r-2d}{r}m+\frac{2d}{r+2} m\rceil\\
&=m-\lfloor \frac{4d}{(r+2)r}m\rfloor\geq m-1,
\end{align*}
which leads to a contradiction that $m-2\geq m_1+m_2\geq m-1$.
This verifies Claim \ref{raboundcD21}.
\end{proof}
From Claim \ref{raboundcD21}, express $2dm=p_1r+p_2(r+2)$ for some non-negative integers $p_1$ and $p_2$ with $p_1<r+2$. Note that \begin{align*}
    \frac{2}{a}m&=\frac{2dm}{ad}=\frac{p_1r+p_2(r+2)}{r+1}=p_1+p_2+\frac{p_2-p_1}{r+1}, \\
    \frac{2d}{r+2}m&=\frac{p_1r+p_2(r+2)}{r+2}=p_1+p_2-\frac{2p_1}{r+2},\\
    \frac{2d}{r}m&=\frac{p_1r+p_2(r+2)}{r}=p_1+p_2+\frac{2p_2}{r}.
\end{align*}
Now the integer $(p_1+p_2)r=2dm-2p_2$ is even where $r$ is odd. This gives that $p_1+p_2$ is even. The inequalities $\lfloor \frac{2}{a}m\rfloor \ge m_2\geq \lceil \frac{2d}{r+2} m\rceil$ then imply $\lfloor \frac{p_2-p_1}{r+1} \rfloor\geq m_2-(p_1+p_2)\geq\lceil -\frac{2p_1}{r+2}\rceil$. From the assumption $p_1<r+2$ and that $m_2-(p_1+p_2)$ is even, one sees $\lfloor \frac{p_2-p_1}{r+1} \rfloor\geq 0$. In particular, $p_2\ge p_1$. 

Similarly, the inequalities $\lfloor \frac{a-2}{a}m\rfloor \ge m_1\ge \lceil \frac{r-2d}{r}m\rceil$ are equivalent to $\lfloor \frac{2d}{r}m\rfloor\geq m-m_1 \geq \lceil\frac{2}{a}m\rceil$. This implies $\lfloor\frac{2p_2}{r}\rfloor\ge m-m_1-(p_1+p_2)\geq \lceil\frac{p_2-p_1}{r+1}\rceil$ where the integer $m-m_1-(p_1+p_2)$ is even. 
In particular, $p_2\ge r$ provided that $p_2-p_1>0$.


\noindent {\bf Case 2.}
There exists an analytical identification: $$P\in X\simeq o\in \left( \begin{array}{ll}
\varphi_{1}:= x^2+yt+p(z^2,u)=0 \\
 \varphi_{2}:= yu+z^{2d+1}+q(z^{2},u)zu+t=0 \\
\end{array} \right)$$ in $\hat{\mathbb{C}}_{x,y,z,u,t}^5/\frac{1}{2}(1,1,1,0,1)$ where $o$ denotes the origin of $\hat{\mathbb{C}}_{x,y,z,u,t}^5/\frac{1}{2}(1,1,1,0,1)$ such that  $\sigma: Y \to X$  is a weighted blow up of weights $w=\frac{1}{2}(r+2,r,a,2,r+4)$ with center $P\in X$ and 
\begin{itemize}
\item $r+2=a(2d+1)$ where $d$, $r$ and $a\geq 5$ are positive integers;
\item $\ct(X,S)=\frac{a}{m}$ where $m=w(f)$ and $S$ is defined by the formal power series $f=0$ analytically and locally.
\end{itemize}
On the open subset $U_2=\{\overline{y}\neq 0\}$, $Y$ is isomorphic to $\hat{\mathbb{C}}_{\overline{x},\overline{y},\overline{z}}^3/\frac{1}{r}(-(r+2),2,-a)$. It follows from terminal lemma that both integers $a$ and $r$ are odd. 
Denote by $\sigma_1\colon Y_1\to X$ and $\sigma_2\colon Y_2\to X$ the weighted blow up with weights $w_{1}=\frac{1}{2}(r-2d+1,r-2d-1,a-1,2,r-2d+3)$ and $w_2=\frac{1}{2}(2d+1,2d+1,1,2,2d+1)$ at the origin $P\in X$ respectively. 
By \cite[Lemma 2.1, Lemma 7.6 and Lemma 7.7]{3ct}, there exist two integers $m_1$ and $m_2$ satisfying 
$$  \lfloor \frac{a-1}{a} m \rfloor  \ge m_1\ge \lceil \frac{r-2d-1}{r} m \rceil \textup{ and } \lfloor \frac{1}{a} m \rfloor  \ge m_2\ge \lceil \frac{2d+1}{r+4} m \rceil $$ and $m_1\equiv (a-1)a^{-1}m$ and $m_2\equiv a^{-1}m$ (mod $2$).

\begin{claim} \label{raboundcD22} If $a\nmid m$, then $(4d+2)m\geq r(r+4)$. 
\end{claim} 
\begin{proof}[Proof of Claim \ref{raboundcD22}]
Since $a\nmid m$, we have 
\[m-1=\lfloor \frac{1}{a}m\rfloor+\lfloor \frac{a-1}{a}m\rfloor\geq m_1+m_2.\]
As $m_1+m_2\equiv m$ (mod 2), one has $m-2\geq m_1+m_2$.

Suppose on the contrary that $(4d+2)m< r(r+4)$.
We see 
\begin{align*}
m_1+m_2&\geq \lceil \frac{2d+1}{r+4} m\rceil +\lceil \frac{r-2d-1}{r}m\rceil \ge   \lceil \frac{2d+1}{r+4} m+ \frac{r-2d-1}{r}m\rceil\\
&=m-\lfloor \frac{8d+4}{(r+4)r}m\rfloor\geq m-1.
\end{align*}
which leads to a contradiction that $m-2\geq m_1+m_2\geq m-1$.
This verifies the claim.
\end{proof}
From Claim \ref{raboundcD22}, express $(4d+2)m=p_1r+p_2(r+4)$ for some non-negative integers $p_1$ and $p_2$ with $p_1<r+4$. 
Note that \begin{align*}
    \frac{1}{a}m&=\frac{(4d+2)m}{(4d+2)a}=\frac{p_1r+p_2(r+4)}{2(r+2)}=\frac{p_1+p_2}{2}+\frac{p_2-p_1}{r+2}, \\
    \frac{2d+1}{r+4}m&=\frac{(4d+2)m}{2(r+4)}=\frac{p_1r+p_2(r+4)}{2(r+4)}=\frac{p_1+p_2}{2}-\frac{2p_1}{r+4},\\
    \frac{2d+1}{r}m&=\frac{(4d+2)m}{2r}=\frac{p_1r+p_2(r+4)}{2r}=\frac{p_1+p_2}{2}+\frac{2p_2}{r}.
\end{align*} 
Now the integer $(p_1+p_2)r=(4d+2)m-4p_2$ is even where $r$ is odd. In particular, two integers $p_1+p_2$ and $p_2-p_1$ are even. The inequalities $\lfloor \frac{1}{a}m\rfloor \ge m_2\geq \lceil \frac{2d+1}{r+4} m\rceil$ then imply $\lfloor \frac{p_2-p_1}{r+2} \rfloor\geq m_2-\frac{p_1+p_2}{2}\geq\lceil -\frac{2p_1}{r+4}\rceil$. From the assumption $p_1<r+4$ and that $$m_2-\frac{p_1+p_2}{2}\equiv a^{-1}m-\frac{p_1+p_2}{2}\equiv \frac{p_2-p_1}{r+2}\equiv p_2-p_1\equiv 0 \textup{ (mod $2$)},$$ one sees $\lfloor \frac{p_2-p_1}{r+2} \rfloor\geq 0$. In particular, $p_2\ge p_1$. 

Similarly, the inequalities $\lfloor \frac{a-1}{a}m\rfloor \ge m_1\ge \lceil \frac{r-2d-1}{r}m\rceil$ are equivalent to $\lfloor \frac{2d+1}{r}m\rfloor\geq m-m_1 \geq \lceil\frac{1}{a}m\rceil$. This implies $\lfloor\frac{2p_2}{r}\rfloor\ge m-m_1-\frac{p_1+p_2}{2}\geq \lceil\frac{p_2-p_1}{r+2}\rceil$ where $$m-m_1-\frac{p_1+p_2}{2}\equiv m-(a-1)a^{-1}m-\frac{p_1+p_2}{2}\equiv \frac{p_2-p_1}{r+2}\equiv 0 \textup{ (mod $2$)}.$$
In particular, $p_2\ge r$ provided that $p_2-p_1>0$.
This completes the proof of Proposition \ref{classfycD2ct}.
\end{proof}

\begin{rem} \label{alternative} We keep notions in Case 1 of the proof of Proposition \ref{classfycD2ct}.
As $w_2\succeq \frac{2d}{r+2}w$, one sees $m_2\geq \lceil \frac{2d}{r+2} m \rceil$. In what follows, we shall provide an alternative argument to establish the inequality  $\lfloor\frac{2}{a} m \rfloor  \ge m_2$. Let $\cX=\bC^4/\frac{1}{2}(1,1,1,0)$ and $\sigma_w: \cY\to \cX$ be the weighted blow up with weights $w=\frac{1}{2}(r+2,r,a,2)$. Put $\overline{\varphi}_1=\varphi(\overline{x}^{\frac{r+2}{2}},\overline{x}^{\frac{r}{2}}\overline{y}, \overline{x}^{\frac{a}{2}}\overline{z},\overline{x}\overline{u})/\overline{x}^{r+1}$. As $w(x^2)=r+2$ and $w(\varphi)=r+1$, $\overline{x}\in \overline{\varphi}_1$ and hence $$Y\cap U_1\simeq (\overline{\varphi}_1=0)/\frac{1}{r+2}(2,-r,-a,-2) \simeq \bC_{\overline{y},\overline{z},\overline{u}}^3/\frac{1}{r+2}(-r,-a,-2),$$ where $U_1:=\{\overline{x}\neq 0\}\simeq \bC^4/\frac{1}{r+2}(2,-r,-a,-2)$ is an open subset of $\cY$. Denote by $\sigma\colon Y\to X$ be the induced morphism with exceptional divisor $E$. Let $\mu:\cZ_1\to \cY$ be the weighted blow up (at the origin of $U_1$) with weights $w_2=\frac{1}{r+2}(2d, 2d, 1, r+2-2d)$. Denote by $Z_1=Y_{\cZ_1}$ the proper transform. Then the induced map $\mu_1=\mu|_{Z_1}:Z_1\to Y$ is Kawamata blow up. In particular, the exceptional set of $\mu_1$, denoted by $E_1$, is a prime divisor of $Z_1$. 
Then we see  
\begin{align*}
 K_Y&=\sigma^* K_X+\frac{a}{2}E, \ \ K_{Z_1}=\mu_1^*K_Y+\frac{1}{r+2}E_1,\\
 S_Y&=\sigma^*S-\frac{m}{2}E, \ \ S_{Z_1}=\mu_1^*S_Y-\frac{m'}{r+2}E_1,
\end{align*}
for some non-negative integer $m'$.
Note that it follows from $w_2=\frac{2d}{r+2}w+\frac{1}{r+2}(0,2d,1,r+2-2d)$ that $E_{Z_1}=\mu_1^*E-\frac{2d}{r+2}E_1,$. By direct toric computations, one has 
\begin{align*}
K_{Z_1}&=\mu_1^*\sigma^*K_X+\frac{a}{2}E_{Z_1}+(\frac{2d}{r+2}\cdot \frac{a}{2}+\frac{1}{r+2})E_1=\mu_1^*\sigma^*K_X+\frac{a}{2}E_{Z_1}+\frac{2}{2}E_1\\
S_{Z_1}&=\mu_1^*\sigma^*S-\frac{m}{2}E_{Z_1}-(\frac{2d}{r+2}\cdot\frac{m}{2}+\frac{m'}{r+2})E_1=\mu_1^*\sigma^*S-\frac{m}{2}E_{Z_1}-\frac{m_2}{2}E_1,
\end{align*}
where $m_2=2w_2(f)$ and the divisor $S$ is defined by the formal power series $f=0$ analytically and locally. As $\ct(X,S)$ is the canonical threshold and $E_1$ is a prime divisor over $X$, $\frac{2}{m_2}\geq \ct(X,S)=\frac{a}{m}$. In particular, $\lfloor\frac{2}{a} m \rfloor  \ge m_2$. 
\end{rem}

\begin{thm}\label{classfy3ct}
We have $\mathcal{T}^{\textup{can}}_{3} = \{0\} \cup \{\frac{4}{5}\}\cup \mathcal{T}_{3, \textup{sm}}^{\textup{can}}.$  \end{thm}
\begin{proof}
We show non-trivial inclusion $\mathcal{T}^{\textup{can}}_{3} \subseteq \{0\} \cup \{\frac{4}{5}\}\cup \mathcal{T}_{3, \textup{sm}}^{\textup{can}}$. Suppose that $
 \ct(X,S)\in \mathcal{T}^{\textup{can}}_{3}$ is non-zero. By taking $\bQ$-factorization, $X$ can be assumed to have at worst $\bQ$-factorial terminal singularities. From the decomposition in (3) in   Remark \ref{quickremark1}, we may assume $\ct(X,S)\in \mathcal{T}^{\textup{can}}_{3,*,\geq 5}$ for some singular type $*=cA, cA/n, cD$ or $cD/2$. The rest follows from Theorem \ref{classfysmct}, Propositions \ref{possible 3ctcA}, \ref{possible  3ctcA/n new},  \ref{classfycDct} and  \ref{classfycD2ct}.
\end{proof}

\begin{rem}\label{quickremark2} 
It is known that the accumulation points of $\mathcal{T}^{can}_{3}$ is 
the set $\{0\} \cup \{\frac{1}{k}\}_{k \in \bZ_{\ge 2}} $ by \cite{3ctpr} and \cite{HLL} independently. This result can be recovered using Theorem \ref{classfy3ct}. It is also interesting to study  $(\frac{1}{k-1}, \frac{1}{k}) \cap C$ for any $k \in \bZ_{\geq 2}$. The $k=2$ case was explicitly characterized in \cite{3ct}. 
The set $\mathcal{T}^{\textup{can}}_{3,\textup{sm}}\cap (\frac{1}{3},\frac{1}{2})$ was studied by the third named author in his master thesis \cite{Wu} and listed in Table \ref{tab:my_label}. These two cases can also be recovered by Theorems \ref{classfysmct} and \ref{classfy3ct}.
\newpage 


\begin{table}[]
    \centering
    \begin{tabular}{|c| c| c| c | l |} 
 \hline
 $\alpha$ & $\beta$ & $p_1$ & $p_2$ & Remark \\ [0.5ex] 
 \hline
 $\alpha$ & $\beta$ & $p_1$ & $3$ & $\begin{array}{l} \ct=\frac{1}{3}+\frac{(3-p_1)\alpha}{3(p_1\alpha+3\beta)}\\
 \textup{with } 0\leq p_1\leq 2, \\ 1\leq \alpha\leq 3, \\  (2-p_1)\alpha<\beta, \\\gcd(\alpha,\beta)=1\end{array}$\\ 
  \hline
 $1$ & $1$ & $1$ & $4$ & $\ct=2/5$ \\
  \hline
 $1$ & $1$ & $0$ & $5$ & $\ct=2/5$ \\
  \hline
 $1$ & $2$ & $0$ & $4$ & $\ct=3/8$ \\
 \hline
 $2$ & $3$ & $0$ & $4$ & $\ct=5/12$ \\
 \hline
$2$ & $3$ & $1$ & $4$& $\ct=5/14$\\
 \hline
$2$ & $5$ & $0$ & $4$ & $\ct=7/20$\\
 \hline
 $3$ & $4$ & $0$ & $4$ & $\ct=7/16$\\
 \hline
 $3$ & $4$ & $1$ & $4$ & $\ct=7/19$\\
 \hline
  $3$ & $4$ & $0$ & $5$ & $\ct=7/20$\\
 \hline
 $3$ & $5$ & $0$ & $4$ & $\ct=2/5 $\\
  \hline
 $3$ & $5$ & $1$ & $4$ & $\ct=8/23 $\\
  \hline
 $3$ & $7$ & $0$ & $4$ & $\ct=5/14 $\\
  \hline
 $3$ & $8$ & $0$ & $4$ & $\ct=11/32$\\
  \hline
 $4$ & $5 $& $0$ & $4$ & $\ct=9/20$\\
  \hline
 $4$ & $5$ & $1$ & $4$ & $\ct=3/8 $\\
  \hline
 $4$ & $5$ & $0$ & $5$ & $\ct=9/25 $\\
  \hline
 $4$ &$7$ & $0$ & $4$ & $\ct=11/28 $\\
   \hline
 $4$ & $7$ & $1$ & $4$ & $\ct=11/32$\\
  \hline
 $4$ & $9$ & $0$ & $4$ & $\ct=13/36 $\\
  \hline
 $4$ & $11$ & $0$ & $4$ & $\ct=15/44 $\\
  \hline
 $5$ & $6$ & $0$ & $5$ & $\ct=11/30 $\\
   \hline
 $5$ & $7$ & $0$ & $5$ & $\ct=12/35 $\\
 \hline  
\end{tabular} 
\caption{canonical thresholds in $\mathcal{T}^{\textup{can}}_{3,\textup{sm}}\cap (\frac{1}{3},\frac{1}{2})$}
    \label{tab:my_label}
\end{table}    
\end{rem}

\end{document}